\documentclass[11pt]{amsart}
\usepackage{amsthm,amssymb}
\usepackage[usenames]{color}
\usepackage{tikz}
\usetikzlibrary{decorations.markings,arrows}
\usepackage{booktabs}
\usepackage{hyperref}

\numberwithin{figure}{section}
\numberwithin{table}{section}
\newtheorem{thm}{Theorem}[section]
\newtheorem{lem}[thm]{Lemma}
\newtheorem{cor}[thm]{Corollary}
\newtheorem{prop}[thm]{Proposition}

\theoremstyle{definition}
\newtheorem{defn}[thm]{Definition}
\newtheorem{construction}[thm]{Construction}
\theoremstyle{remark}
\newtheorem{remark}[thm]{Remark}
\newtheorem*{remark*}{Remark}

\newcommand{\link}[2]{\ensuremath{\mathrm{Lk}(#1,#2)}}
\newcommand{\field}[1]{\ensuremath{\mathbb{#1}}}
\newcommand{\real}{\field{R}}
\newcommand{\integer}{\field{Z}}
\newcommand{\complex}{\field{C}}
\newcommand{\str}[1]{\ensuremath{\Sigma_{#1}}}
\newcommand{\eps}{\varepsilon}
\newcommand{\simg}{\sim}
\newcommand{\simh}{\simeq}
\newcommand{\linkedtriangles}[1]{\ensuremath{\mathcal{T}_{#1}}}

\newlength{\hgap}
\newlength{\vgap}
\newlength{\radius}

\thanks{The second author was supported in part by NSF Grant DMS-1607744}

\begin{document}
\title[Weakly linked embeddings of complete graphs in $\mathbb{R}^3$]{Weakly linked embeddings of pairs of complete graphs in $\mathbb{R}^3$}

\author[Di]{James Di}
\address{ByteDance, 250 Bryant Street, Mountain View, CA, 94041, USA }
\email{jamesyubai.di@gmail.com}

\author[Flapan]{Erica Flapan}
\address{Department of Mathematics, Pomona College, Claremont, CA 91711, USA}
\email{eflapan@pomona.edu}

\author[Johnson]{Spencer Johnson}
\address{Austin Community College, 5930 Middle Fiskville Rd, Austin, Texas 78752, USA}
\email{spencer.johnson@austincc.edu}

\author[Thompson]{Daniel Thompson}
\address{Washington Post, 1301 K St NW, Washington, DC, 20005, USA}
\email{dwt101092@gmail.com}

\author[Tuffley]{Christopher Tuffley}
\address{School of Fundamental Sciences, Massey University,
         Private Bag 11 222, Palmerston North 4442, New Zealand}
\email{c.tuffley@massey.ac.nz}


\begin{abstract}
Let $G$ and $H$ be disjoint embeddings of complete graphs $K_m$ and $K_n$ in $\real^3$ such that some cycle in $G$ links a cycle in $H$ with non-zero linking number. We say that $G$ and $H$ are \emph{weakly linked} if the absolute value of the linking number of any cycle in $G$ with a cycle in $H$ is $0$ or $1$.   Our main result is an algebraic characterisation of when a pair of disjointly embedded complete graphs is weakly linked. 

As a step towards this result, we show that if $G$ and $H$ are weakly linked, then each contains either a vertex common to all triangles linking the other or a triangle which shares an edge with all triangles linking the other. All families of weakly linked pairs of complete graphs are then characterised by which of these two cases holds in each complete graph.
\end{abstract}

\maketitle

\section{Introduction}

The study of linked cycles within an embedded graph began in 1983 with Conway and Gordon's \cite{conway-gordon1983} and Sachs' \cite{sachs1983} result that every embedding of $K_6$ in $\mathbb{R}^3$ contains a pair of triangles with non-zero linking number.  Any graph with this property is said to be \emph{intrinsically linked}.  In the same paper, Sachs showed that each of the seven graphs in the Petersen family is intrinsically linked and no minor of any of them is intrinsically linked.  Then, in 1995, Robertson, Seymour, and Thomas \cite{RST} proved that these seven graphs are the only graphs which are minor minimal with respect to being intrinsically linked.  Since then, many results have been obtained about intrinsic linking of graphs.

In this paper, we explore how pairs of cycles in disjointly embedded complete graphs in $\mathbb{R}^3$ can be linked.  We consider linking from a purely algebraic point of view.  Thus we say that disjoint simple closed curves $C$ and $D$ are \emph{linked} if and only if their pairwise linking number $\link{C}{D}$ is non-zero. We introduce the following definitions.

\begin{defn}
We say that disjointly embedded simple closed curves $C$ and $D$ in $\real^3$ are \emph{strongly linked} if $|\link{C}{D}|\geq 2$, 
and \emph{weakly linked} if $|\link{C}{D}|=1$.\end{defn}

\begin{defn}
We say that disjointly embedded graphs $G$ and $H$ in $\real^3$ are \emph{strongly linked} if some cycle in $G$ strongly links a cycle in $H$; and \emph{weakly linked} if some cycle in $G$ links a cycle in $H$, but no cycle in $G$ strongly links any cycle in $H$.
\end{defn}

Our main result is a characterisation of all weakly linked embeddings of $G\cong K_m$ and $H\cong K_n$ in terms of the pairwise linking numbers between triangles in $G$ and triangles in $H$.  Since any cycle in a complete graph can be decomposed as a sum of triangles, this completely determines all pairwise linking numbers between cycles in $G$ and cycles in $H$. 

We build our results in stages as follows. In Section~\ref{sec:SCC}, we prove Theorem~\ref{thm:star}, which characterises
 weak linking between a simple closed curve and an embedded complete graph $K_n$. Since the complete graph $K_3$ is a cycle, this also characterises  weak linking of $K_m$ and $K_n$ when $\min\{m,n\}=3$.  In Section~\ref{sec:theta}, we prove Theorem~\ref{thm:thetacurve}, which characterises  weak linking between a theta curve (i.e., a graph with two vertices joined by three edges, homeomorphic to the Greek letter $\Theta$) and a complete graph $K_n$. Next, in Section~\ref{sec:K4Kn}, we prove Theorem~\ref{thm:K4Kn}, which characterises  weak linking of $K_4$ and $K_n$, for $n\geq4$.  In Section~\ref{sec:nca}, we prove Theorem~\ref{thm:nocommonapex}, which is a technical result needed for our characterisation of weakly linked embeddings of $K_m$ and $K_n$.  Finally in Section~\ref{sec:KmKn}, we prove the following dichotomy. 

\begin{thm}[Theorem~\ref{thm:common-vertex-or-triangle} paraphrased]
\label{thm:dichotomy}
Let $m\geq 5$ and $n\geq 4$, and suppose that $G\cong K_m$ and $H\cong K_n$ are weakly linked in $\mathbb{R}^3$.  Then exactly one of the following holds:
\begin{enumerate}
\item
There is a vertex $p$ of $G$ common to all triangles of $G$ linking $H$ (``$G$ contains a common vertex'').
\item
There is a triangle $T^*$ in $G$ such that a triangle $T\neq T^*$ of $G$ links $H$ if and only if it shares an edge with $T^*$ (``$G$ contains a common triangle'').
\end{enumerate}
\end{thm}
Then in Theorem~\ref{thm:KmKn-triangle}, we characterise weak linking between $G$ and $H$ when at least one of $G$ and $H$ contains no vertex common to all triangles linking the other; while in Theorem~\ref{thm:KmKn-pq}, we characterise weak linking when both $G$ and $H$ contain a vertex common to all triangles linking the other.  

The concept of a \emph{star} (defined below) will play a key role in our results.

\begin{defn}
Let $\bigl(\{p\},O,I\bigr)$ be an ordered partition of the vertices of $K_n$, where $O=\{q_1,\dots,q_k\}$ and $I=\{r_1,\dots,r_\ell\}$. The \emph{star} $pOI$ consists of all oriented triangles of the form $pqr$, where $q\in O$ and $r\in I$. We also express the star $pOI$ as $p|q_1\cdots q_k|r_1\cdots r_\ell$. 

The vertex $p$ is said to be the \emph{apex} of the star. A star $pOI$ is \emph{proper} if neither $O$ nor $I$ is a singleton, and \emph{improper} otherwise.  
Note that the improper stars $p\{q\}I$ and $qI\{p\}$ are equal.
We also refer to an improper star $p\{q\}I$ as a \emph{fan} with \emph{axis} $pq$. 

If $\Sigma=pOI$ is a star, then we define $-\Sigma$ to be the star $-\Sigma=pIO$. We say that $-\Sigma$ is obtained by reversing the orientation of $\Sigma$. 
\end{defn}

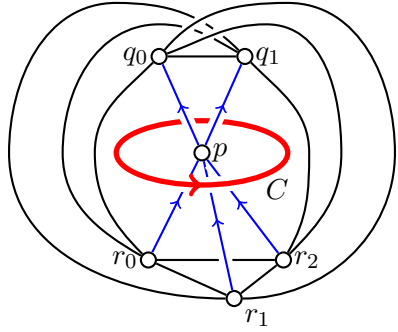
\begin{figure}[t]
\begin{center}
\setlength{\hgap}{1cm}
\setlength{\vgap}{1cm}
\begin{tikzpicture}[vertex/.style={circle,draw,minimum size = 2mm,inner sep=0pt},thick,scale=0.57]
\node (p) at (0,0) [vertex] {};
\node (0) at (0,0) {};
\node (q0) at (-\hgap,2.25*\vgap) [vertex] {};
\node (q1) at (\hgap,2.25*\vgap) [vertex] {};
\node (r0) at (-1.25*\hgap,-2.5*\vgap) [vertex] {};
\node (r1) at (0.75*\hgap,-3.4*\vgap) [vertex] {};
\node (r2) at (1.9*\hgap,-2.5*\vgap) [vertex] {};
\foreach \j/\k in {0/1,1/2,2/0} 
    \draw (r\j) -- (r\k);
\foreach \j in {0,1,2} 
   {\draw [color=white,line width=5] (0) -- (r\j);
    \draw [color=blue,decoration={markings,mark=at position 0.5 with {\arrow{<}}},postaction={decorate}] (p) -- (r\j);}
\draw (0,0) [color=white,line width=6] ellipse (2cm and 0.75cm);
\draw (0,0)[decoration={markings, mark=at position 0.95 with \node [below] {$C$};},postaction={decorate}] ellipse (2cm and 0.75cm);
\draw (0,0)[color=red,line width=2,decoration={markings, mark=at position 0.75 with {\arrow{>}}},postaction={decorate}] ellipse (2cm and 0.75cm);
\foreach \j in {0,1} 
   {\draw [color=white,line width=5] (p) -- (q\j);
     \draw [color=blue,decoration={markings,mark=at position 0.5 with {\arrow{>}}},postaction={decorate}] (p) -- (q\j);}
\draw (q0) -- (q1);
%q1-r0
\draw (q1) to [out=150,in=0] (-1,3) to [out=180,in=90] (-3.25,0) to [out=-90,in=150] (r0);
%q1r0xq0r2
\draw [fill,color=white] (0.15*\hgap,2.73*\vgap) circle (0.125cm);
%q0r1xq1r0
\draw [fill,color=white]  (-0.32*\hgap,2.91*\vgap) circle (0.125cm);
%q1-r1
\draw (q1) to [out=135,in=0] (-2,3.5)  to [out=180,in=90] (-4.5,0) to [out=-90,in=180] (r1);
%q0r2xq1r1
\draw [fill,color=white] (0.39*\hgap,2.825*\vgap) circle (0.125cm);
%q0-r2
\draw (q0) to [out=20,in=180] (1.5,3) to [out=0,in=90] (3.25,0) to [out=-90,in=45] (r2);
%q0r1xq1r1
\draw [fill,color=white] (-0.05,3.07*\vgap) circle (0.125cm);
%q0-r1
\draw (q0) to [out=50,in=180] (2,3.5) to [out=0,in=90] (4.5,0) to [out=-90,in=0] (r1);
%q0-r0
\draw (q0) to [out=-135,in=90] (-2.5,0) to [out=-90,in=135] (r0);
%q1-r2
\draw (q1) to [out=-45,in=90] (2.5,0) to [out=-90,in=60] (r2);
\foreach \v/\p/\l in {p/right/p,q0/left/q_0,q1/right/q_1}
     \node [\p] at (\v) {$\l$};
\foreach \r/\p in {0/{left},1/{below right},2/{right}} 
     \node [\p] at (r\r) {$r_\r$};
\end{tikzpicture}
\caption{An embedding of $K_6=\langle p,q_0,q_1,r_0,r_1,r_2\rangle$ and a curve $C$ such that $C$ links $K_6$ in the star $p|q_0q_1|r_0r_1r_2$. }
\label{fig:K6star}
\end{center}
\end{figure}

\begin{defn}
Let $n\geq3$, and let $C$ be an oriented simple closed curve disjoint from an embedding of $K_n$ in $\mathbb{R}^3$. We say that {\em $C$ links $K_n$ in the star $pOI$} if for all oriented triangles $T$ of $K_n$ we have
\[
\link{C}{T}=\begin{cases}
             +1 & \text{if $T$ is a triangle of the star $pOI$}, \\
             -1 & \text{if $-T$ is a triangle of the star $pOI$}, \\\
              0 & \text{otherwise}.
           \end{cases}
\]
\end{defn}

Figure~\ref{fig:K6star} shows an example of an 
embedding of $K_6=\langle p,q_0,q_1,r_0,r_1,r_2\rangle$ and a curve $C$ which links $K_6$ in the star $p|q_0q_1|r_0r_1r_2$.  The vertex $p$ together with all incident edges (shown in blue) form a star, hence the name.

Unoriented stars with $\min\{|O|,|I|\}\leq2$ were previously used by Flapan, Naimi, and Pommersheim~\cite{I3L} and Drummond-Cole and O'Donnol~\cite{drummondcole-odonnol2009} to study intrinsically $n$-linked graphs.  In particular, a graph $G$ is said to be \emph{intrinsically $n$-linked} or I$n$L if every embedding of $G$ in $\real^3$ contains a non-split link of $n$-components. Flapan, Naimi, and Pommersheim used stars to prove that $K_{10}$ is I3L; and then Drummond-Cole and O'Donnol used them to show that for every $n\geq 2$, $K_{\lfloor\frac{7}{2}n\rfloor}$ is I$n$L.

\section{Weak linking of a simple closed curve with $K_n$}
\label{sec:SCC}
The main result of this section is the following theorem, which shows that weak linking between a simple closed curve and a complete graph can be characterised in terms of stars.

\begin{thm}
\label{thm:star}
Let $n\geq 3$, and let $C$ be an oriented simple closed curve disjoint from an embedding of $K_n$ in $\real^3$ such that $C$ links some cycle of $K_n$. Then $C$ weakly links $K_n$ if and only if $C$ links $K_n$ in a star. 
\end{thm}

The case $n=3$ is immediate, using any vertex as the apex and the remaining two vertices as $O$ and $I$. For $n\geq4$ we prove Theorem~\ref{thm:star} as a series of lemmas, beginning with the ``if'' direction in Lemma~\ref{lem:star-implies-weak}:

\begin{lem}
\label{lem:star-implies-weak}
Let $C$ be an oriented cycle disjoint from an embedding of $K_n$ in $\mathbb{R}^3$. If $C$ links $K_n$ in a star, then $C$ weakly links $K_n$.
\end{lem}

\begin{proof}
Suppose that $C$ links the star $pOI$ in $K_n$, and let $D=v_0v_1\cdots v_{k-1}$ be a $k$-cycle in $K_n$. We first show that if $D$ does not contain $p$, then $D$ does not link $C$.

To do this, decompose $D$ as the sum of the triangles $T_i=v_0v_iv_{i+1}$, for $1\leq i\leq k-2$, so that in the homology group $H_1(\mathbb{R}^3-C)$ we have
\[
[D] = \sum_{i=1}^{k-2} [T_i].
\]
Then since $C$ links $K_n$ in the star $pOI$, and $D$ does not contain $p$, we have $[T_i]=0$ for all $i$. Therefore $[D]=0$, showing that $D$ does not link $C$.

Suppose now that $D$ does contain $p$. Since $C$ links $K_n$ in the star $pOI$, $C$ does not strongly link any triangle in $K_n$.  Thus we may assume that $k\geq 4$.
Assume without loss of generality that $v_{0}=p$, and
let $T=v_0v_1v_{k-1}$ and $D'=v_1v_2\ldots v_{k-1}$. Then $T$ is a triangle, $D'$ is a $(k-1)$-cycle, and $D=T+D'$  as $1$-chains in $K_n$. The cycle $D'$ does not contain $p$ so by the previous paragraph, in $H_1(\mathbb{R}^3-C)$ we have
\[
[D]=[T]+[D']=[T]\in\{0,\pm1\}.
\]
Therefore $D$ does not strongly link $C$.  Since $C$ links $K_n$, it follows that $C$ weakly links $K_n$, as required.
\end{proof}

In order to prove the ``only if'' direction of Theorem~\ref{thm:star}, we first prove the case $n=4$ in Lemma~\ref{lem:K4case}; then we use Lemma~\ref{lem:K4case} to prove the case $n=5$ in Lemma~\ref{lem:K5case}; then finally we use Lemma~\ref{lem:K5case} to prove the case $n\geq6$ in Lemma~\ref{lem:Kmgeq6}.  

\begin{lem}
\label{lem:K4case}
Let $C$ be an oriented simple closed curve which weakly links an embedding of $K_4$ in $\mathbb{R}^3$.  Then $C$ links $K_4$ in a fan. 
\end{lem}

\begin{proof}
Let $K_4=\langle v_0,v_1,v_2,v_3\rangle$, and let
\begin{align*}
C_0 &= v_1v_2v_3, & C_1 &= v_3v_2v_0, \\
C_2 &= v_0v_1v_3, & C_3 &= v_2v_1v_0.
\end{align*}
Then as $1$-chains in $K_4$ we have
\[
C_0+C_1+C_2+C_3 = 0,
\]
and for $i\neq j$ the sum $C_i+C_j$ is a 4-cycle in $K_4$. 

In the homology group $H_1(\real^3-C)$ we have
\[
[C_0]+[C_1]+[C_2]+[C_3] = 0,
\]
with each $[C_i]\in\{0,\pm1\}$  and some $[C_i]\not=0$. If there exist $i\neq j$ such that $[C_i]=[C_j]\neq0$, then $[C_i+C_j]=2[C_i]\neq0$, and $C$ strongly links the four cycle $C_i+C_j$, contrary to hypothesis. So it must be the case that one term is equal to $+1$, one term is equal to $-1$, and the other two are zero. After relabelling the vertices and reorienting $C$ (if necessary), we may assume that 
\begin{align*}
[C_0]&=[C_1]=0, & [C_2]&=-[C_3]=1. 
\end{align*}
Thus we let $O=\{v_1\}$, $I=\{v_2,v_3\}$, and see that $C$ links $K_4$ in the fan $v_0\{v_1\}I$. 
\end{proof}

\begin{lem}
\label{lem:K5case}
Let $C$ be an oriented simple closed curve which weakly links an embedding of $K_5$ in $\mathbb{R}^3$.  Then $C$ links $K_5$ in a star. 
\end{lem}

\begin{proof} 
Since any cycle that links $C$ can be broken into triangles, there must be at least one triangle in $K_5$ that links $C$.  First we suppose that there is some edge $pq$ common to all triangles which link $C$.  
Let $K_5=\langle p,q,r_0,r_1,r_2\rangle$, and assume without loss of generality that $\link{C}{pqr_0}=+1$. We claim that $C$ links $K_5$ in the star $p|q|r_0r_1r_2$. 

To see this, we apply Lemma~\ref{lem:K4case} to $K_4=\langle p,q,r_0,r_i\rangle$ for $i=1,2$. By Lemma~\ref{lem:K4case}, $C$ links $K_4$ in a star.  This star must be the fan with axis $pq$ because we know $\link{C}{pqr_0}=+1$ and $pq$ is common to all triangles in $K_5$ linking $C$. Therefore for $i=1,2$, the triangle $pqr_i$ in $K_4=\langle p,q,r_0,r_i\rangle$ links $C$ with linking number $+1$. Thus every triangle in the star $p|q|r_0r_1r_2$ in $K_5$ links $C$ positively; every triangle in the star $q|p|r_0r_1r_2$ links $C$ negatively; and since every triangle that links $C$ contains $pq$, no other triangle can link $C$. It follows that $C$ links $K_5$ in the star $p|q|r_0r_1r_2$ as claimed, completing the proof in this case.

Suppose now that there is no edge of $K_5$ common to all triangles linking $C$. Since any two triangles in $K_5$ must share at least one vertex, this implies there exist triangles $T_0=pq_0r_0$ and $T_1=pq_1r_1$ such that $T_0\cap T_1=\{p\}$ and $\link{C}{T_0}=\link{C}{T_1}=+1$. We show that $C$ links $K_5$ in the star $\Sigma=p|q_0q_1|r_0r_1$.

In what follows homology classes are taken with respect to $H_1(\real^3-C)$, and subscripts are taken modulo 2. We begin by showing that the two remaining triangles $pq_0r_1$ and $pq_1r_0$ of $\Sigma$ link $C$ with linking number $+1$. To see that $[pq_ir_{i+1}]=+1$ for $i=0,1$ consider the 5-cycle $D=pr_iq_ir_{i+1}q_{i+1}$. We have
\begin{align*}
[D]=[pr_iq_ir_{i+1}q_{i+1}]&=[pr_iq_i]+[pq_ir_{i+1}]+[pr_{i+1}q_{i+1}] \\
                           &= [pq_ir_{i+1}]-2,
\end{align*}
so we must have $[pq_ir_{i+1}]=+1$ because otherwise either $D$ or $pq_ir_{i+1}$ would strongly link $C$. 

We next show that $[pq_0q_1]=[pr_0r_1]=0$. Recall that $q_i$ and $r_i$ were chosen so that $[pq_0r_0]=+1$ and $[pq_1r_1]=+1$.  Suppose that $[pq_iq_{i+1}]=+1$ for some $i\in\{0,1\}$. Then letting $D=pq_iq_{i+1}r_{i+1}$ we have
\begin{align*}
[D]=[pq_iq_{i+1}r_{i+1}]&=[pq_iq_{i+1}]+[pq_{i+1}r_{i+1}]=+2.
\end{align*}
Similarly, if $[pr_ir_{i+1}]=+1$ for some $i$, then letting $D=pq_ir_ir_{i+1}$ we have
\begin{align*}
[D]=[pq_ir_ir_{i+1}]&=[pq_ir_{i}]+[pr_{i}r_{i+1}]=+2.
\end{align*}
In either case, some cycle in $K_5$ would strongly link $C$, contrary to hypothesis. Since $i$ can be either $0$ or $1$, we must have
$[pq_0q_1]=[pr_0r_1]=0$. 

To complete the proof that $C$ links $K_5$ in the star $\Sigma=p|q_0q_1|r_0r_1$, it remains to show that $C$ links no triangle in $K_4=\langle q_0,q_1,r_0,r_1\rangle$. Suppose to the contrary that it does. Then it must positively link some triangle of the form $q_iq_{i+1}r_j$ or $q_ir_jr_{j+1}$. In the first case, letting $D=q_iq_{i+1}r_jp$ we have
\[
[D]=[q_iq_{i+1}r_jp] = [q_iq_{i+1}r_j]+[q_{i+1}r_jp]=+2;
\]
and in the second, letting $D=q_ir_jr_{j+1}p$ we similarly get
\[
[D] = [q_ir_jr_{j+1}p] = [q_ir_jr_{j+1}]+[q_ir_{j+1}p]=+2.
\]
In either case some cycle in $K_5$ strongly links $C$, contrary to hypothesis. So no triangle in $K_4=\langle q_0,q_1,r_0,r_1\rangle$ can link $C$, and we conclude that $C$ links $K_5$ in the star  $\Sigma=p|q_0q_1|r_0r_1$. This completes the proof.
\end{proof}

\begin{lem}
\label{lem:Kmgeq6} 
Let $n\geq6$, and 
let $C$ be an oriented simple closed curve which weakly links an embedding of $K_n$ in $\mathbb{R}^3$.  Then $C$ links $K_n$ in a star. 
\end{lem}

To prove Lemma~\ref{lem:Kmgeq6} we will use Lemma~\ref{lem:tripleimpliesstrong}, which is a special case of a lemma proved in Flapan~\cite{flapan2002}.

\begin{lem}[Triple link implies strong link]
\label{lem:tripleimpliesstrong}
Let $L\cup Z\cup W$ be a $3$-component link in $\real^3$, such that $Z$ and $W$ are cycles belonging to an embedding of $K_n$ in $\real^3$. Suppose that $\link{L}{Z}\neq0\neq\link{L}{W}$. Then $K_n$ contains a cycle which strongly links $L$. 
\end{lem}

\begin{proof}
In Lemma~1 of ~\cite{flapan2002} the component $L$ is also assumed to be a cycle belonging to $K_n$, but this hypothesis plays no role in the proof and can be omitted. If either $Z$ or $W$ strongly links $L$ then we are done. Otherwise, we may orient $Z$ and $W$ such that $\link{L}{Z}=\link{L}{W}=1$, and apply \cite[Lemma 1]{flapan2002} to obtain a cycle $J$ in $K_n$ with at least $6$ vertices such that for some orientation of $J$, we have $\link{L}{J}\geq 2$.
\end{proof}

\begin{proof}[Proof of Lemma~\ref{lem:Kmgeq6}]  
Since $C$ links some cycle of $K_n$, it links some triangle.  Since $n>5$, this triangle lies in a $K_5$ subgraph which links $C$.  Thus by Lemma~\ref{lem:K5case}, $C$ links a star in this $K_5$.  It follows that $C$ links at least three triangles in $K_n$.  If there are two disjoint triangles in $K_n$ which link $C$, then by Lemma~\ref{lem:tripleimpliesstrong} $C$ strongly links some cycle of $K_n$.  So we assume that no pair of triangles that link $C$ are disjoint.

Now we show that there is a vertex $p$ in $K_n$ such that every triangle in $K_n$ that links $C$ contains $p$.  Suppose this is not the case.  Since no pair of triangles that link $C$ are disjoint, there must exist three triangles  linking $C$ which pairwise intersect but don't all share a common vertex.  We know by Lemma~\ref{lem:K5case} that the set of triangles in a $K_5$ which link $C$ must all share at least one common vertex.  Thus the three triangles which pairwise intersect but don't have a common vertex must use at least $6$ vertices.  If any pair of them shared a common edge, it would only require $5$ vertices; and if they used more than $6$ vertices, there would be a pair that did not share a vertex.  Thus we have the situation illustrated in Figure~\ref{fig:pairwiseintersect}, with $C$ linking the triangles $p_0p_1q_2$, $p_0q_1p_2$, and $q_0p_1p_2$.

\begin{figure}[t]
\begin{center}
\setlength{\hgap}{1.25cm}
\begin{tikzpicture}[vertex/.style={circle,draw,minimum size = 2mm,inner sep=0pt},thick,scale=0.5]
\foreach \j/\p/\q in {0/below/above,1/right/left,2/left/right}
  {\node (x\j) at (-90+\j*120:\hgap) [vertex] {};
   \node [\p] at (x\j) {$p_\j$};
   \node (y\j) at (90+\j*120:2*\hgap) [vertex] {};
   \node [\q] at (y\j) {$q_\j$};}
\foreach \j/\k in {0/1,1/2,2/0}
   {\draw (y\j) -- (x\k);
    \draw (x\j) -- (y\k);
    \draw (x\j) -- (x\k);}
\end{tikzpicture}
\caption{Three triangles $p_0p_1q_2$, $p_0q_1p_2$, $q_0p_1p_2$ that pairwise intersect but share no common vertex.}
\label{fig:pairwiseintersect}
\end{center}
\end{figure}
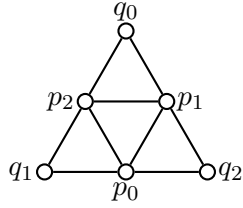

 Suppose without loss of generality that $\link{C}{p_0p_1q_2}=+1$ (re-orienting $C$, if necessary), and consider the $K_5$-subgraph $H=\langle p_0,p_1,p_2,q_1,q_2\rangle$. Since $p_0p_1q_2$ and $p_0q_1p_2$ both link $C$ and are contained in $H$, $C$ must link $H$ in a proper star with apex $p_0$.  This star must be either $p_0|p_1q_1|p_2q_2$ (if $\link{C}{p_0q_1p_2}=+1$), or $p_0|p_1p_2|q_1q_2$ (if $\link{C}{p_0q_1p_2}=-1$). Hence either $p_0p_1p_2$ and $p_0q_1q_2$ both link $C$, or $p_0p_1q_1$ and $p_0p_2q_2$ both link $C$.

If $p_0p_1p_2$ and $p_0q_1q_2$ both link $C$, then the triangles $p_0q_1q_2$ and $q_0p_1p_2$ in $K_n$ would be disjoint triangles which both link $C$.  Hence by Lemma~\ref{lem:tripleimpliesstrong}, there would be a cycle in $K_n$ strongly linking $C$. As this is contrary to hypothesis, we must have both
$p_0p_1q_1$ and $p_0p_2q_2$ linking $C$ instead. Then in the $K_5$-subgraph $H'=\langle p_0,p_1,p_2,q_0,q_1\rangle$ we have at least the three triangles $p_0q_1p_2$, $q_0p_1p_2$, and  $p_0p_1q_1$ linking $C$.  But this is impossible by Lemma~\ref{lem:K5case} since $C$ must link $H'$ in a star, which means that there is a vertex common to all triangles in $H'$ which link $C$.  Thus, in fact, there must be some vertex $p$ in $K_n$ such that every triangle in $K_n$ that links $C$ contains $p$. 

We next show that every vertex $s$ of $K_n$ belongs to some triangle in $K_n$ linking $C$. Indeed, let $pqr$ be a triangle in $K_n$ linking $C$, and consider the $K_4$-subgraph $\langle p,q,r,s\rangle$. By Lemma~\ref{lem:K4case} this $K_4$-subgraph must link $C$ in a fan, so either $pqs$ or $prs$ must link $C$. 

We're now ready to partition the vertices of $K_n-\{p\}$ into sets $O$ and $I$ as required by the theorem. 
Let 
\begin{align*}
O &= \{q\in K_n:\text{$[pqr]=1$ for some $r\in K_n$}\}, \\
I &= \{r\in K_n:\text{$[pqr]=1$ for some $q\in K_n$}\},
\end{align*}
where homology classes are taken with respect to $H_1(\real^3-C)$.
By the previous paragraph every vertex of $K_n-\{p\}$ belongs to $O\cup I$. 
We claim that moreover $O\cap I=\emptyset$, so that $\{O,I\}$ is in fact a partition of the vertices of $K_n-\{p\}$.
To see this, suppose that $r\in O\cap I$. Then there are vertices $q$ and $s$ such that $[pqr]=[prs]=1$.  But this would imply that
\[
[pqrs]=[pqr]+[prs]=2,
\]
and hence the square $pqrs$ would strongly link $C$.  Thus $O\cap I=\emptyset$.

 Given $q\in O$ and $r\in I$, we must show that $C$ links $pqr$.  Now by definition of $O$ and $I$, there are vertices $s\in I$ and $t\in O$ such that $[pqs]=[ptr]=1$. If $s=r$ or $t=q$, then $C$ does link $pqr$ as required.  Otherwise $p,q,r,s,t$ are all distinct so $H=\langle p,q,r,s,t\rangle$ is a $K_5$-subgraph. By  Lemma~\ref{lem:K5case}, $C$ must link $H$ in the star $p|qt|rs$, and so $[pqr]=1$, as required. 

Finally, to show that $C$ only links triangles of the form $pqr$ with $q\in O$ and $r\in I$, we consider triangles in $K_n$ which are not of this form. Since $\bigl\{\{p\},O,I\bigr\}$ partitions the vertices of $K_n$, such a triangle must have one of the following forms:
\begin{itemize}
\item
$pq_1q_2$ with $q_1,q_2\in O$, which cannot link $C$ since that would imply that $q_1$ or $q_2$ belongs to $O\cap I$.

\item
$pr_1r_2$ with $r_1,r_2\in I$, which cannot link $C$ since that would imply that $r_1$ or $r_2$ belongs to $O\cap I$.

\item
$xyz$ with $p\notin\{x,y,z\}$, which cannot link $C$ because we showed that every triangle that links $C$ contains $p$.
\end{itemize}
So we are done.
\end{proof}

Taken together, Lemmas~\ref{lem:K4case},~\ref{lem:K5case}, and~\ref{lem:Kmgeq6} complete the proof of Theorem~\ref{thm:star}.  By Lemma~\ref{lem:tripleimpliesstrong}, if a cycle $C$ weakly links a complete graph $K_n$ then any two cycles in $K_n$ that link $C$ must intersect. We extend this result 
to a pair of weakly linked complete graphs as follows.

\begin{thm}
\label{thm:notdisjoint}
Let $m,n\geq 3$, and suppose that $G\cong K_m$ and $H\cong K_n$ are weakly linked graphs in $\real^3$. If $C_1$ and $C_2$ are cycles in $G$ that link $H$, then there is a vertex $p$ of $G$ that belongs to both $C_1$ and $C_2$. 
\end{thm}

To prove Theorem~\ref{thm:notdisjoint}, we first prove the following lemma.

\begin{lem}
\label{lem:commonlinkingcycle} 
Let $C_1$ and $C_2$ be oriented simple closed curves which weakly link an embedding of $K_n$ in $\mathbb{R}^3$. Then there is a cycle of length at most $4$ in $K_n$ that links both $C_1$ and $C_2$.
\end{lem}

\begin{proof}
By Theorem~\ref{thm:star}, $C_1$ and $C_2$ each link $K_n$ in a star. Let $q_1O_1I_1$ be the star that links $C_1$, and let $q_2O_2I_2$ be the star that links $C_2$.  We show as follows that there is a cycle in $K_n$ that links both $C_1$ and $C_2$.

We first suppose that $q_1\not = q_2$.  Since $\bigl\{\{q_1\},O_1,I_1\bigr\}$ and $\bigl\{\{q_2\}, O_2, I_2\bigr\}$ each partition $K_n$, we can switch the orientations on $C_1$ and $C_2$ (if necessary) so that $q_1\in I_2$ and $q_2\in O_1$.  Hence for all $r\in I_1$, we have $\link{C_1}{rq_1q_2}=1$; and for every $r\in O_2$, we have $\link{C_2}{rq_1q_2}=1$.  If there is some $r\in I_1\cap O_2$, then the triangle $q_1q_2r$ links both $C_1$ and $C_2$.  

 So we assume that $I_1\cap O_2=\emptyset$.  Since 
$\bigl\{\{q_1\},O_1,I_1\bigr\}$ and 
$\bigl\{\{q_2\},O_2,I_2\bigr\}$ are partitions of $K_n$ with $q_1\in I_2$ and 
$q_2\in O_1$, it follows that we must have 
$O_2\subseteq O_1$ and 
 $I_1\subseteq I_2$.  
Let $x\in I_1\subseteq I_2$, $y\in O_2\subseteq O_1$, and consider the square
$xq_1q_2y$. 
In $H_1(\mathbb{R}-C_1)$ we have
\[
[xq_1q_2y]=[xq_1y]+[yq_1q_2]=1+0=1,
\]
because $q_2\in O_1$; and in $H_1(\mathbb{R}-C_2)$ we have 
\[
[xq_1q_2y]=[xq_1q_2]+[xq_2y]=0+1=1,
\]
because  $q_1\in I_2$. Thus $C_1$ and $C_2$ both link the square $xq_1q_2y$.  

Next suppose that $q_1=q_2$.  If $O_2\subseteq O_1$, then by analogy with our above argument $I_1\subseteq I_2$.  In this case, if $x\in O_2$ and $y\in I_1$, then $q_1xy$ links both $C_1$ and $C_2$.   Thus we assume that $O_2\not \subseteq O_1$, and similarly that $O_1\not \subseteq O_2$.  It follows that there is some vertex $x\in O_2\cap I_1$ and some vertex $y\in I_2\cap O_1$.  But now $xyq_1$ links both $C_1$ and $C_2$.  
\end{proof}

\begin{proof}[Proof of Theorem~\ref{thm:notdisjoint}]
By hypothesis $C_1$ and $C_2$ both weakly link $H$, so by  Lemma~\ref{lem:commonlinkingcycle} there is a cycle $D$ in $H$ which links both $C_1$ and $C_2$. 
Theorem~\ref{thm:notdisjoint} now follows from either
Lemma~\ref{lem:tripleimpliesstrong} or 
Theorem~\ref{thm:star} applied to the cycle $D$ and the complete graph $G$:
for instance, by Theorem~\ref{thm:star} $D$ links $G$ in a star $pOI$, and then $p$ must belong to $C_i$ for $i=1,2$, because otherwise $C_i$ does not link $D$.
\end{proof}

\section{Weak linking of a $\Theta$ curve with $K_n$}
\label{sec:theta}

\begin{thm}
\label{thm:thetacurve}
Let $\Theta$ be a theta curve with oriented cycles $C_1,C_2,C_3$ such that
\[
[C_1]+[C_2]+[C_3]=0
\]
in $H_1(\Theta)$. 
Let $n\geq3$, and suppose that $\Theta$ and $K_n$ are weakly linked graphs in $\real^3$. Then exactly one of the following cases holds:
\begin{enumerate}
\renewcommand{\labelenumi}{(A\arabic{enumi})}
\renewcommand{\theenumi}{A\arabic{enumi}}
\item
\label{case:commonapex}
There is a vertex $p$ of $K_n$ common to all triangles linking a curve in $\Theta$. Then there are pairwise disjoint sets $I_1,I_2,I_3$ (at most one empty) such that $I_1\cup I_2\cup I_3=K_n-\{p\}$, and after reversing the orientation of $\Theta$ (if necessary), each $C_i$ links the star $pO_iI_i$ in $K_n$, where
\begin{align*}
O_1 &= I_2\cup I_3, &
O_2 &= I_1\cup I_3, &
O_3 &= I_1\cup I_2.
\end{align*}
\item
\label{case:nocommonapex}
There is no vertex of $K_n$ common to all triangles linking a curve in $\Theta$. Then $n\geq 5$ and there are distinct vertices $p_1,p_2,p_3$ in $K_n$ such that, after reversing the orientation of $\Theta$ (if necessary), each $C_i$ links the star $p_iO_iI$ in $K_n$, where
\begin{align*}
O_1 &= \{p_2,p_3\}, &
O_2 &= \{p_1,p_3\}, &
O_3 &= \{p_1,p_2\},
\end{align*}
and $I=K_n-\{p_1,p_2,p_3\}$.
\end{enumerate}
\end{thm}

Figure~\ref{fig:thetacurve} illustrates the two cases. The loops $C_1$, $C_2$ and $C_3$ link stars as given in the theorem. It then follows from Theorem~\ref{thm:star} that these embeddings are weakly linked because every cycle in $\Theta$ links a star in $K_n$. 

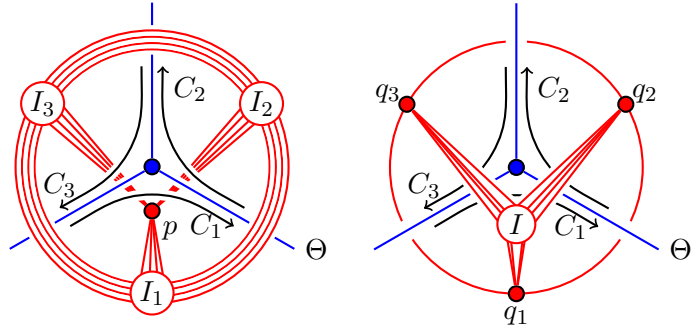
\begin{figure}
\begin{center}
\setlength{\radius}{2.1cm}
\begin{tabular}{cc}
\begin{tikzpicture}[vertex/.style={circle,draw,minimum size = 2mm,inner sep=0pt},thick,scale=0.8]
%  invisible line to align the two pictures
\draw [white] (90:1.3*\radius) -- (-90:1.32*\radius);
% blue centre
\node (b) at (0,0) [vertex,fill=blue] {};
% red central vertex
\node (q) at (0,-0.35*\radius) [vertex,fill=red] {};
\node [below right] at (q) {$p$};
% apices and other things
\foreach \i in {1,2,3}
   {\coordinate (b\i) at (90+\i*120:1.3*\radius);
    \path (b) ++ (96+\i*120:0.8*\radius) coordinate (s\i);
    \path (b) ++ (150+\i*120:0.22*\radius) coordinate (m\i);
    \path (b) ++ (204+\i*120:0.8*\radius) coordinate (e\i);
    \coordinate (p\i) at (150+\i*120:\radius);}
% red lumped vertices
% centres
\foreach \i/\t in {1,2,3}    
   {\node (q\i) at (150+\i*120:\radius) {};}
% vertices at q1
\foreach \j/\m/\t in {0/1/15,1/1/{-10},2/1/-5,3/1/0}
    \path (q1) ++ (-14+\t+\j*90:\m*0.11*\radius) coordinate (q1\j);
% vertices at q2
\foreach \j/\m/\t in {0/1/40,1/1/{-10},2/1/0,3/1.1/10}
    \path (q2) ++ (-14+\t+\j*90:\m*0.11*\radius) coordinate (q2\j);
% vertices at q3
\foreach \j/\m/\t in {0/1/15,1/1/{-10},2/1/0,3/0.8/0}
    \path (q3) ++ (-34+\t+\j*90:\m*0.11*\radius) coordinate (q3\j);
%% red edges
% draw each in red
\foreach \i in {1,2,3}
    {\foreach \j in {0,1,2,3}
        \draw [color=red] (q) -- (q\i\j);}
% blue edges, black arrows
\foreach \i in {1,2,3}
   {\draw [color=white,line width=4] (s\i) to [out=270+\i*120,in=60+\i*120] (m\i) to [out=240+\i*120,in=30+\i*120] (e\i);
    \draw [->] (s\i) to [out=270+\i*120,in=60+\i*120] (m\i) to [out=240+\i*120,in=30+\i*120] (e\i);
    \draw [color=white,line width=4] (b) -- (b\i); 
    \draw [color=blue] (b) -- (b\i);}
% labels on arrows
\foreach \i/\p in {1/left,2/below right,3/above}
        \node [\p] at (e\i) {$C_\i$};
% red circles
\foreach \i in {0,1,2,3}
     {\draw [color=white,line width=4] (0,0) circle (0.93*\radius+\i*0.05*\radius);
      \draw [color=red] (0,0) circle (0.93*\radius+\i*0.05*\radius);}
% red lumping circles
\foreach \i in {1,2,3}
   {\filldraw [color=white,draw=red] (q\i) circle (0.17*\radius);
    \node at (q\i) {$I_\i$};}
% label on theta
\node [right]  at (b2) {$\Theta$};
\end{tikzpicture} &
%%%%%%%%%%%%%%%%%%%%%%%%%%%%%%%%%%%%%%%%%%%%%%%%%%%%%%%%%%%%%%%
\begin{tikzpicture}[vertex/.style={circle,draw,minimum size = 2mm,inner sep=0pt},thick,scale=0.8]
%  invisible line to align the two pictures
\draw [white] (90:1.3*\radius) -- (-90:1.25*\radius);
% red circle
\draw [color=red] (0,0) circle (\radius);
% blue centre
\node (b) at (0,0) [vertex,fill=blue] {};
% apices and other things
\foreach \i in {1,2,3}
   {\coordinate (b\i) at (90+\i*120:1.3*\radius);
    \draw [color=white,line width=8] (b) -- (b\i); 
    \draw [color=blue] (b) -- (b\i); 
    \path (b) ++ (96+\i*120:0.8*\radius) coordinate (s\i);
    \path (b) ++ (150+\i*120:0.22*\radius) coordinate (m\i);
    \path (b) ++ (204+\i*120:0.8*\radius) coordinate (e\i);
    \node (p\i) at (150+\i*120:\radius) [vertex,fill=red] {};
    \node at (150+\i*120:1.175*\radius) {$q_\i$};}
% black arrows
\foreach \i in {1,2,3}
  \draw [->] (s\i) to [out=270+\i*120,in=60+\i*120] (m\i) to [out=240+\i*120,in=30+\i*120] (e\i);
% labels on arrows
\foreach \i/\p in {1/left,2/below right,3/above}
    \node [\p] at (e\i) {$C_\i$};
% lumped red vertices
\begin{scope}[yshift=-0.45*\radius]
\foreach \j/\m/\t in {0/1.2/{-10},1/1.1/{-5},2/1/3,3/1/0}
    \coordinate (y\j) at (-55+\t+\j*90:\m*0.1*\radius);
\coordinate (redcentre) at (0,0);
\end{scope}
%% red edges
% draw each first in thick white
\foreach \i in {1,2,3}
    {\foreach \j in {0,1,2,3}
        \draw [color=white,line width=4] (p\i) -- (y\j);}
% draw each in red
\foreach \i in {1,2,3}
    {\foreach \j in {0,1,2,3}
        \draw [color=red] (p\i) -- (y\j);}
% red lumping circle
\filldraw [color=white,draw=red] (redcentre) circle (0.15*\radius);
\node at (redcentre) {$I$};
% label on theta
\node [right]  at (b2) {$\Theta$};
\end{tikzpicture}
\end{tabular}
\caption{Cases~\eqref{case:commonapex} (left) and~\eqref{case:nocommonapex} (right) of Theorem~\ref{thm:thetacurve}. The second vertex of $\Theta$ is placed at infinity.}
\label{fig:thetacurve}
\end{center}
\end{figure}

\begin{proof}
Let $D$ be a cycle in $K_n$ linking some cycle $C_i$ in $\Theta$. Then in the homology group $H_1(\real^3-D)$ we have
\[
[C_1]+[C_2]+[C_3]=0.
\] 

Since there is no strong link between $\Theta$ and $K_n$, every $[C_i]$ is $\pm 1$ or $0$. Since some term is non-zero, each of the three possible values must occur exactly once in the sum. Thus, $D$ must link exactly two of the $C_i$, one positively, and one negatively. We will use this fact repeatedly.

Without loss of generality $C_1$ links some cycle in $K_n$. Since $C_1$ does not strongly link any cycle in $K_n$, by Theorem~\ref{thm:star} it links some star $p_1O_1I_1$. Let $p_1qr$ be a triangle in $p_1O_1I_1$. Then $p_1qr$ links exactly one of $C_2$ and $C_3$.  So without loss of generality we may assume that $C_2$ links $p_1qr$, and therefore $C_2$ links a star $p_2O_2I_2$. We consider two cases, according to whether or not there is a vertex $p$ common to all triangles in $K_n$ linking either $C_1$ or $C_2$. 

\subsection*{Case 1: All triangles linking $C_1$ or $C_2$ share a vertex}

Suppose that there is a vertex $p$ common to all triangles linking either $C_1$ or $C_2$. Then we may choose the stars linking $C_1$ and $C_2$ so that $p_1=p_2=p$. Moreover, since any triangle linking $C_3$ must also link either $C_1$ or $C_2$, if $C_3$ also links some triangle it contains the vertex $p$.  Hence in this case, $C_3$ must link a star of the form $pO_3I_3$. 

Observe that $O_1\cap O_2,O_1\cap I_2,I_1\cap O_2,I_1\cap I_2$ are disjoint sets with union $X=K_n-\{p\}$. Given a triangle $pqr$ in $K_n$
we have
\[
\link{pqr}{C_3} = -\link{pqr}{C_1}-\link{pqr}{C_2},
\]
so the ordered triple 
\[
\bigl(\link{pqr}{C_1},\link{pqr}{C_2},\link{pqr}{C_3}\bigr) 
\] 
is completely determined by the sets in $O_1\cap O_2,O_1\cap I_2,I_1\cap O_2,I_1\cap I_2$ that $q$ and $r$ belong to. Calculating these triples we obtain Table~\ref{tab:thetalinkingtriples}.

If $O_1\cap O_2$ and $I_1\cap I_2$ are both nonempty then by Table~\ref{tab:thetalinkingtriples} $C_3$ strongly links some cycle in $K_n$. Since this is contrary to our hypothesis, at least one of these intersections must be empty. Reversing the orientation of $\Theta$ switches the roles of $O_i$ and $I_i$ for each $i$, so after doing this (if necessary), without loss of generality we may assume that $I_1\cap I_2=\emptyset$. Table~\ref{tab:thetalinkingtriples} then shows that $\link{pqr}{C_3}=+1$ precisely when $q\in(O_1\cap I_2)\cup(I_1\cap O_2)$ and $r\in O_1\cap O_2$. Observe that the conditions $I_1\cap I_2=\emptyset$ and $O_1\cup I_1=O_2\cup I_2=X$ together imply $I_1\subseteq O_2$, $I_2\subseteq O_1$. We therefore have
\begin{align*}
O_3 &= I_1\cup I_2, &
I_3 &= O_1\cap O_2,
\end{align*}
and we see that $I_1$, $I_2$, $I_3$ are pairwise disjoint sets with union $X$. This implies that we also have 
\begin{align*}
O_1&=I_2\cup I_3, &
O_2&=I_1\cup I_3,
\end{align*}
and it follows that $C_1$, $C_2$ and $C_3$ link $K_n$ in stars as given by~\eqref{case:commonapex}.

\begin{table}
\[
\begin{array}{ccccc}\toprule
   & \multicolumn{4}{c}{q} \\ \cmidrule{2-5}
r   & O_1\cap O_2 & O_1\cap I_2 & I_1\cap O_2 &  I_1\cap I_2 \\ \midrule
O_1\cap O_2 & (0,0,0)   & (0,-1,+1) & (-1,0,+1) & (-1,-1,+2) \\
O_1\cap I_2 & (0,+1,-1) & (0,0,0)   & (-1,+1,0) & (-1,0,+1)  \\
I_1\cap O_2 & (+1,0,-1) & (+1,-1,0) & (0,0,0)   & (0,-1,+1)  \\
I_1\cap I_2 & (+1,+1,-2)& (+1,0,-1) & (0,+1,-1) & (0,0,0)  \\ \bottomrule
\end{array}
\]
\caption{The triples $\bigl(\link{pqr}{C_1},\link{pqr}{C_2},\link{pqr}{C_3}\bigr)$ in Case~1 of the proof of Theorem~\ref{thm:thetacurve}.}
\label{tab:thetalinkingtriples}
\end{table}

\subsection*{Case 2: Triangles linking $C_1$ or $C_2$ do not all share a vertex}

Every triangle linking $C_1$ contains $p_1$, and every triangle linking $C_2$ contains $p_2$, so if these triangles do not all share a common vertex we must have $p_1\neq p_2$. Reversing the orientation of $\Theta$ exchanges the roles of $O_1$ and $I_1$, so without loss of generality we may assume that $p_2\in O_1$. At the beginning of the proof, we assumed that without loss of generality there is a triangle $T$ in the star $p_1O_1I_1$ linking $C_2$.  Now since every triangle linking $C_2$ contains $p_2$, $T$ must have the form $p_1p_2r$ for some $r\in I_1$. Then 
\begin{equation}
\label{eq:bothoutvertices}
\link{p_2p_1r}{C_2} = -\link{p_2p_1r}{C_1} = \link{p_1p_2r}{C_1}=+1.
\end{equation}
It follows that we also must have $p_1\in O_2$.

Now there must be some triangle in $K_n$ linking $C_1$ that does not contain $p_2$, because otherwise $p_2$ would be common to all triangles linking $C_1$ or $C_2$. Let $T_1=p_1q_1r_1\in p_1O_1I_1$ be such a triangle. Then $T_1$ does not link $C_2$ because it does not contain $p_2$, so it must link $C_3$ instead. By Theorem~\ref{thm:star} $C_3$ links a star $p_3O_3I_3$ in $K_n$, where $p_3\in T_1$. If $p_3=p_1$ then this vertex would be common to all triangles linking $C_1$ or $C_3$. Since every triangle linking $C_2$ also links either $C_1$ or $C_3$, this would give us a vertex common to all triangles linking $C_1$ or $C_2$, which is contrary to the hypothesis of this case. So $p_3\neq p_1$. We also have $p_3\neq p_2$, because $p_2\notin T_1$.  So the vertices $p_1$, $p_2$ and $p_3$ are distinct.

Suppose now that $p_3=r_1$, so that $p_3\in I_1$. Then there must be some vertex $r\in I_1$ such that $r\not =r_1$, because otherwise $p_3$ would be common to all triangles linking $C_1$ or $C_3$, and hence to all triangles linking $C_1$ or $C_2$. Consider the triangle $p_1q_1r$. This triangle links $C_1$, because it belongs to the star $p_1O_1I_1$, but it does not link either $C_2$ or $C_3$, because it does not contain either $p_2$ or $p_3$. This is a contradiction, so we must have $p_3=q_1\in O_1$. The argument of equation~\eqref{eq:bothoutvertices} then gives $p_1\in O_3$.

There must also be some triangle $T_2=p_2q_2r_2\in p_2O_2I_2$ that does not contain $p_1$, because otherwise $p_1$ is common to all triangles linking $C_1$ or $C_2$. Arguing as above we conclude that $p_3\in O_2$, and $p_2\in O_3$. We now have
\begin{align}
\label{eq:O-inclusions}
\{p_2,p_3\}&\subseteq O_1, &
\{p_1,p_3\}&\subseteq O_2, &
\{p_1,p_2\}&\subseteq O_3.
\end{align}
Suppose that there is some $q\in O_1$ such that $p_2\neq q\neq p_3$. Then for any $r\in I_1$ the triangle $p_1qr$ links $C_1$, because it belongs to the star $p_1O_1I_1$, but it does not link $C_2$ or $C_3$ because it does not contain $p_2$ or $p_3$. This is a contradiction, so we must have $O_1=\{p_2,p_3\}$. By the same reasoning the other inclusions in ~\eqref{eq:O-inclusions} must also be equalities, and we obtain finally
\begin{align*}
O_1 &= \{p_2,p_3\}, &
O_2 &= \{p_1,p_3\}, &
O_3 &= \{p_1,p_2\},
\end{align*}
and hence
\[
I_1=I_2=I_3=I=K_n-\{p_1,p_2,p_3\}
\] 
as claimed. To conclude we note that we must have $n\geq 5$, because if $n=4$ we would have $I=\{r\}$, and the vertex $r$ would be common to all triangles linking a cycle in $\Theta$.
\end{proof}

\begin{cor}
\label{cor:disjointaxes}
Let $n\geq 4$. 
With notation as in Theorem~\ref{thm:thetacurve}, let $\Theta$ and $K_n$ be disjointly embedded in $\real^3$ such that 
\begin{enumerate}
\item
$C_1$ links a fan in $K_n$ with axis $pq$, and
\item
$C_2$ links a star in $K_n$ that is either a fan with axis $xy$ disjoint from $pq$, or a proper star with apex $x$ disjoint from $pq$.
\end{enumerate}
Then $C_3$ strongly links some cycle in $K_n$.
\end{cor}

\begin{proof}
Since $n\geq 4$ the only vertices common to all triangles linking $C_1$ are $p$ and $q$, and the only vertices common to all triangles linking $C_2$ are either $x$ and $y$ (if $C_2$ links a fan), or $x$ alone (if $C_2$ links a proper star). By hypothesis there is therefore no vertex common to all triangles linking $\Theta$, and so if $\Theta$ does not strongly link $K_5$ we must be in Case~\eqref{case:nocommonapex} of Theorem~\ref{thm:thetacurve}. But in Case~\eqref{case:nocommonapex} of the theorem the stars of $C_1$, $C_2$ and $C_3$ are all proper, contradicting the fact that $C_1$ links a fan. Thus in fact neither case holds, so some cycle in $\Theta$ strongly links a cycle in $K_n$. Both $C_1$ and $C_2$ link stars, so it is $C_3$ that strongly links $K_n$.
\end{proof}

\section{Weak linking of $K_4$ with $K_n$}
\label{sec:K4Kn}

Let $G=\langle p_0,p_1,p_2,p_3\rangle\cong K_4$, and let
\begin{align*}
C_0 &= p_1p_2p_3, & C_2 &= p_0p_1p_3, \\
C_1 &= p_3p_2p_0, & C_3 &= p_2p_1p_0.
\end{align*}
With these orientations we have
\[
[C_0]+[C_1]+[C_2]+[C_3] = 0
\]
in $H_1(K_4;\integer)$, and for any $i\neq j$ the chain $C_i+C_j$ represents a $4$-cycle. We use $C_1$, $C_2$, $C_3$, and $C_4$ in Theorem~\ref{thm:K4Kn}.

\begin{thm}
\label{thm:K4Kn}
Let $n\geq 4$, and 
suppose that $G\cong K_4$ and $H\cong K_n$ are weakly linked graphs in $\real^3$. Then exactly one of the following holds:
\begin{enumerate}
\renewcommand{\labelenumi}{(B\arabic{enumi})}
\renewcommand{\theenumi}{B\arabic{enumi}}
\item
\label{case:K4Kn-commonvertex}
There is a vertex $q$ of $H$ common to all triangles linking a curve in $G$. Then there are pairwise disjoint sets $I_0$, $I_1$, $I_2$, $I_3$ (at most two of them empty) such that $I_0\cup I_1\cup I_2\cup I_3=H-\{q\}$, and after reversing the orientation of $\real^3$ (if necessary), each $C_i$ links the star $qO_iI_i$ in $K_n$, where
\[
O_i = H-\{q\}-I_i.
\]
\item
\label{case:K4Kn-nocommonvertex}
There is no vertex of $H$ common to all triangles linking a curve in $G$. Then $n\geq 5$ and there are distinct vertices $q_1,q_2,q_3$ in $H$ such that, after relabelling the vertices of $G$ and reversing orientation of $\real^3$ (if necessary), $C_0$ does not link $H$ and $C_1,C_2,C_3$ link $H$ in the stars
\[
q_1\{q_2,q_3\}I, \qquad q_2\{q_1,q_3\}I, \qquad q_3\{q_1,q_2\}I,
\]
where $I=H-\{q_1,q_2,q_3\}$. In particular, the vertex $p_0$ of $G$ is common to all triangles of $G$ linking $H$; and a triangle $T$ of $H$ links $G$ if and only if it shares exactly one edge with $T^*=q_1q_2q_3$. 
\end{enumerate}
\end{thm}

Embeddings realising~\eqref{case:K4Kn-commonvertex} and~\eqref{case:K4Kn-nocommonvertex} are illustrated in Figure~\ref{fig:K4Kn}.
These embeddings belong to the families of embeddings realising Theorem~\ref{thm:KmKn-triangle}, which we prove are weakly linked in Theorem~\ref{thm:KmKn-triangle-weak}.

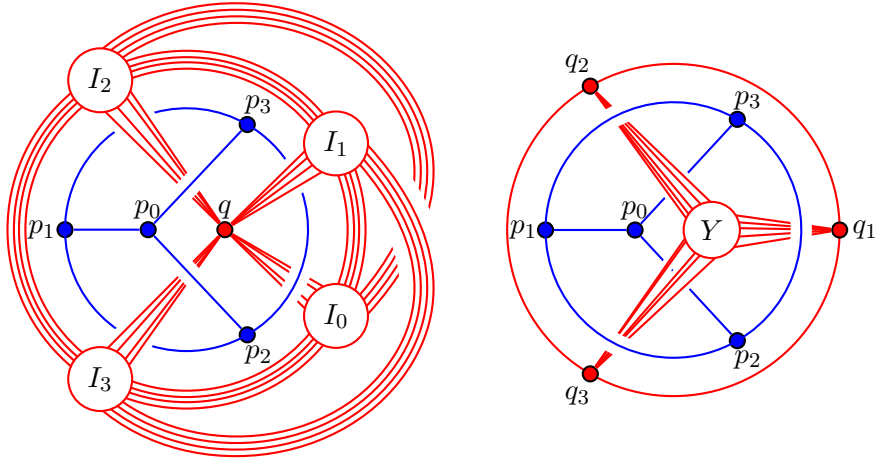
\begin{figure}[t]
\setlength{\radius}{1.7cm}
\begin{center}
\begin{tabular}{cc}
% K4-Kn with a common vertex in Kn
\begin{tikzpicture}[vertex/.style={circle,draw,minimum size = 2mm,inner sep=0pt},thick,null/.style={minimum size = 0mm, inner sep = 0pt}]
% blue apices
\foreach \i/\p in {1/below right,2/right,3/right}
   {\node (p\i) at (60+\i*120:0.95\radius) [vertex,fill=blue] {};
%        \node [\p] at (p\i) {$p_\i$};}
        \node at (60+\i*120:1.13*\radius) {$p_\i$};}
\node (p0) at (-0.3*\radius,0) [vertex,fill=blue] {};
%\node [above] at (p0) {$p_0$};
% red apex
\node (q) at (0.3*\radius,0) [vertex,fill=red] {}; 
% red lumped vertices
% centres
\foreach \i/\t in {0/-30,1/30,2/120,3/240}    
   {\node (q\i) at (\t:1.35*\radius) {};}
% vertices at q0
\foreach \j/\m/\t in {0/1/15,1/1/{-10},2/1/-5,3/1/0}
    \path (q0) ++ (-14+\t+\j*90:\m*0.15*\radius) node (q0\j) [vertex] {};
% vertices at q1
\foreach \j/\m/\t in {0/1/40,1/1/{-10},2/1/0,3/1.1/10}
    \path (q1) ++ (-14+\t+\j*90:\m*0.15*\radius) node (q1\j) [vertex] {};
% vertices at q2
\foreach \j/\m/\t in {0/1/15,1/1/{-10},2/1/0,3/1/0}
    \path (q2) ++ (-34+\t+\j*90:\m*0.15*\radius) node (q2\j) [vertex] {};
% vertices at q3
\foreach \j/\m/\t in {0/1/15,1/1/{-10},2/1/0,3/1/0}
    \path (q3) ++ (-14+\t+\j*90:\m*0.15*\radius) node (q3\j) [vertex] {};
%%% crossing arcs
% q2 to q0 - anchors and arcs
\foreach \i in {0,1,2,3}
   {\coordinate (q2a\i) at (120:1.375*\radius+\i*0.05*\radius);
    \coordinate (q1t\i) at (30:1.9*\radius+\i*0.05*\radius);
    \coordinate (q0a\i) at (-30:1.375*\radius+\i*0.05*\radius);
    \draw [red] (q2a\i) to [out=45,in = 120] (q1t\i) to [out=-60,in=45] (q0a\i);}
% q3 to q1 - anchors
\foreach \i in {0,1,2,3}
   {\coordinate (q3a\i) at (-120:1.375*\radius+\i*0.05*\radius);
    \coordinate (q0t\i) at (-30:1.9*\radius+\i*0.05*\radius);
    \coordinate (q1a\i) at (30:1.375*\radius+\i*0.05*\radius);}
% white blocking arcs
    \draw [white,line width=16] (q3a1) to [out=-45,in = -120] (q0t1) to [out=60,in=-45] (q1a1);
    \draw [white,line width=16] (q3a2) to [out=-45,in = -120] (q0t2) to [out=60,in=-45] (q1a2);
% red arcs
\foreach \i in {0,1,2,3}
  \draw [red] (q3a\i) to [out=-45,in = -120] (q0t\i) to [out=60,in=-45] (q1a\i);
% red circles
\foreach \r in {1.35,1.4,1.26,1.31}
  \draw [color=red] (0,0) circle (\r*\radius);
% red under edges 
%\foreach \j in {0,1,2,3,}
%        \draw [color=white,line width=6] (q) -- (q0\j);
\foreach \j in {0,1,2,3}
        \draw [color=red] (q) -- (q0\j);
%% blue arcs
%draw overarc in white first
\draw [color=white,line width=8] (-40:0.95*\radius) arc (-40:0:\radius);
% then blue arcs
\foreach \i in {1,2,3}
    \draw [color=blue] (p\i) arc (60+\i*120:180+\i*120:0.95\radius);
%% red edges
% draw each first in thick white
\foreach \i in {1,2,3}
    {\foreach \j in {0,1,2,3,}
        \draw [color=white,line width=6] (q) -- (q\i\j);}
% draw each in red
\foreach \i in {1,2,3}
    {\foreach \j in {0,1,2,3}
        \draw [color=red] (q) -- (q\i\j);}
% blue edges
% draw each first in thick white
\foreach \i in {1,2,3}
        \draw [color=white,line width=8] (p\i) -- (p0);
% draw each in blue
\foreach \i in {1,2,3}
        \draw [color=blue] (p\i) -- (p0);
% red lumping circles
\foreach \i in {0,1,2,3} 
{\filldraw [color=white,draw=red] (q\i) circle (0.25*\radius);
\node at (q\i) {$I_{\i}$};}
% label on red apex
\node [above] at (q) {$q$};
% label on blue apex
\node [above] at (p0) {$p_0$};
\end{tikzpicture}&
%%%%%%%%%%%%%%%%%%%%%%%%%%%%%%%%%%%%%%%%%%%%%%%%%%%%%%%%%%%%%%%%
% K4-Kn, no common vertex in H
\begin{tikzpicture}[vertex/.style={circle,draw,minimum size = 2mm,inner sep=0pt},thick]
% red circle
\draw [color=red] (0,0) circle (1.3*\radius);
% apices
\foreach \i in {1,2,3}
   {\node (p\i) at (60+\i*120:\radius) [vertex,fill=blue] {};
    \node (q\i) at (-120+\i*120:1.3*\radius) [vertex,fill=red] {};
    \node at (60+\i*120:1.17*\radius) {$p_\i$};
    \node at (-120+\i*120:1.5*\radius) {$q_\i$};}
% blue central vertex
\node (p0) at (-0.3*\radius,0) [vertex,fill=blue] {};
\node [above] at (p0) {$p_0$};
% lumped red vertices
\begin{scope}[xshift=0.3*\radius]
\foreach \j/\m/\t in {0/1.2/{-10},1/1.1/{-5},2/1/3,3/1/0,4/0.5/0}
    \coordinate (y\j) at (-44+\t+\j*72:\m*0.15*\radius);
\coordinate (redcentre) at (0,0);
\end{scope}
% blue edges
\foreach \i in {1,2,3}
        \draw [color=blue] (p\i) -- (p0);
%% red edges
% draw each first in thick white
\foreach \i in {1,2,3}
    {\foreach \j in {0,1,2,3,4}
        \draw [color=white,line width=4] (q\i) -- (y\j);}
% draw each in red
\foreach \i in {1,2,3}
    {\foreach \j in {0,1,2,3,4}
        \draw [color=red] (q\i) -- (y\j);}
%% blue arcs
% draw each first in thick white
\foreach \i in {1,2,3}
    \draw [color=white,line width=8] (70+\i*120:\radius) arc (70+\i*120:170+\i*120:\radius);
% draw each in blue
\foreach \i in {1,2,3}
    \draw [color=blue] (p\i) arc (60+\i*120:180+\i*120:\radius);
% red lumping circle
\filldraw [color=white,draw=red] (redcentre) circle (0.22*\radius);
\node at (redcentre) {$Y$};
% invisible point to align pictures
\node (xyz) at (-90:2.05*\radius) {};
\end{tikzpicture}
\end{tabular}
\caption{Embeddings of $G\cong K_4$ (blue) and $H\cong K_n$ (red) realising Cases~\eqref{case:K4Kn-commonvertex} (left) and~\eqref{case:K4Kn-nocommonvertex} (right) of Theorem~\ref{thm:K4Kn}.}
\label{fig:K4Kn}
\end{center}
\end{figure}

\begin{proof}
We consider the two cases in turn.

\subsection*{Case 1: All triangles in $H$ linking $G$ share a common vertex $q$.}

Then each triangle $C_i$ in $G$ links a star $qO_iI_i$ in $H$, where we allow the possibility that $O_i$ or $I_i$ is empty to cover the case where $C_i$ doesn't link $H$.

Suppose that $C_i$ and $C_j$ both link $H$. 
If $O_i\cap O_j\neq \emptyset \neq I_i\cap I_j$, then we may choose $x\in O_i\cap O_j$ and $y\in I_i\cap I_j$. Then $C_i$ and $C_j$ both link the triangle $qxy\in H$ with linking number 1, and consequently the square $C_i+C_j$ strongly links $qxy$. So we must have either $O_i\cap O_j=\emptyset$, or $I_i\cap I_j=\emptyset$. Reversing the orientation of $\real^3$ exchanges the roles of $O_\ell$ and $I_\ell$ for all $\ell$, so after doing this (if necessary) we may assume that $I_i\cap I_j=\emptyset$. 
This implies $I_i\subseteq O_j$ and $I_j\subseteq O_i$, because $\{O_\ell,I_\ell\}$ is a partition of $H-\{q\}$ for each $\ell$. 

We claim now that if $C_k$ also links $H$, then $I_i\cap I_k=I_j\cap I_k=\emptyset$. Suppose to the contrary that $I_i\cap I_k$ is nonempty. Then we must have $O_i\cap O_k=\emptyset$, by the previous paragraph, and arguing as above we must have $O_i\subseteq I_k$, and $O_k\subseteq I_i$. But then $I_j\subseteq O_i\subseteq I_k$ and $O_k\subseteq I_i\subseteq O_j$, so $I_j\cap I_k\neq\emptyset \neq O_j\cap O_k$, giving us a strong link. We must therefore have $I_a\cap I_b=\emptyset$ whenever $C_a$ and $C_b$ link $H$, and we extend this to hold for all $a$ and $b$ by setting $I_\ell=\emptyset$, $O_\ell = H-\{p\}$ if $C_\ell$ does not link $H$. Note that at most two of the $I_\ell$ can be empty, because any triangle in $H$ linking $G$ must link it in exactly two triangles, one positively and one negatively.

To complete the proof in this case we must show that $I_0\cup I_1\cup I_2\cup I_3=H-\{q\}$. Let $x\in H-\{q\}$. If $x\notin I_i$ then we necessarily have $x\in O_i$. Then $qxy$ links $C_i$ with linking number $+1$ for some $y$, so it must link some $C_k$ with linking number $-1$. Then $qyx\in qO_kI_k$, and we conclude that $x\in I_k$.

\subsection*{Case 2: There is no vertex common to all triangles in $H$ linking $G$.}

 First we show that some triangle $C_i$ of $G$ must link a proper star in $H$.

Suppose to the contrary that every triangle of $G$ that links $H$ links it in a fan. We may suppose that some triangle $C_i$ links $H$ in the fan with axis $xy$. By assumption $x$ does not belong to every triangle of $H$ linking $G$, so some $C_j$ links a triangle $T_1\subseteq H$ that does not contain $x$. Then $C_j$ links a fan in $H$, and since $C_i\cup C_j$ is a theta curve the axis of this fan must meet $xy$, by Corollary~\ref{cor:disjointaxes}. Thus $C_j$ links $H$ in a fan with axis $\pm yz$, for some $z\neq x$. Now since $y$ is not common to all triangles of $H$ linking $G$, some cycle $C_k$ must link a triangle $T_2$ that does contain $y$. Then $C_k$ links $H$ in a fan, and by Corollary~\ref{cor:disjointaxes} the axis of this fan must meet both $xy$ and $yz$. The axis must therefore be $\pm zx$. But now the triangle $xyz\subseteq H$ links all three triangles $C_i,C_j,C_k$, contradicting the fact that it must link $G$ in a star, which contains exactly two triangles. So some triangle in $G$ must link a proper star in $H$. Note that this immediately implies $n\geq 5$. 

Without loss of generality $C_1$ links a proper star with apex $q_1$. By assumption $q_1$ is not common to all triangles of $H$ linking $G$, so without loss of generality $C_2$ links a triangle $T$ that does not contain $q_1$. Now $C_1\cup C_2$ is a theta curve, and since $q_1$ is the only vertex common to all triangles linking $C_1$, and some triangle linking $C_2$ does not contain $q_1$, there is no vertex common to all triangles of $H$ linking $C_1$ or $C_2$. We must therefore be in Case~\eqref{case:nocommonapex} of Theorem~\ref{thm:thetacurve}, so there are vertices $q_2$ and $q_3$ such that (after reversing orientation of $\real^3$, if necessary) $C_1$, $C_2$ and $-(C_1+C_2)$ link $H$ in the stars
\begin{align*}
\str{C_1} &= q_1\{q_2,q_3\}I, \\
\str{C_2} &= q_2\{q_1,q_3\}I, \\
\str{-(C_1+C_2)} &= q_3\{q_1,q_2\}I,
\end{align*}
where $I=H-\{q_1,q_2,q_3\}$. 

We now consider $C_0$ and $C_3$. 
At least one of them must link $H$, because otherwise $C_0+C_3=-(C_1+C_2)$ would not, a contradiction. After relabelling the vertices of $G$ (if necessary) we may therefore assume that $C_3$ links $H$. We will show under these conditions that $C_0$ does not link $H$. 
To do this we use the fact that $C_0$, $C_3$ and $-(C_0+C_3)=C_1+C_2$ form a theta curve, with $-(C_0+C_3)$ linking $H$ in the star $-\str{-(C_1+C_2)} = q_3I\{q_1,q_2\}$. 

We first show that $C_3$ must link $H$ in a star with apex $q_3$. Suppose to the contrary that $q_3$ is not common to all triangles linking $C_3$.
If $C_3$ links $H$ in a star with apex $q_1$ then by  Case~\eqref{case:nocommonapex} of Theorem~\ref{thm:thetacurve} it must link $H$ in the star $q_1I\{q_2,q_3\}$; but then for any $r\in I$ both $C_2$ and $C_3$ link the triangle $q_2q_1r$ with linking number $+1$, and it follows that the square $C_2+C_3$ strongly links $H$. The same argument shows that $C_3$ cannot link a star with apex $q_2$, so suppose finally that $C_3$ links $H$ in a star with apex $q_4\in I$. Then by Theorem~\ref{thm:thetacurve} Case~\eqref{case:nocommonapex} it must be the case that $|I|=2$, so $n=5$  and there is a vertex $q_5$ such that $C_3$, $C_0$ and $-(C_0+C_3)$ link $H$ in the stars
\begin{align*}
\str{C_0} &= q_5\{q_3,q_4\}\{q_1,q_2\}, \\
\str{C_3} &= q_4\{q_3,q_5\}\{q_1,q_2\}, \\
\str{-(C_0+C_3)} &= q_3\{q_4,q_5\}\{q_1,q_2\}.
\end{align*}
Observe now that $C_1\cup C_3$ is a theta curve, and the stars 
\begin{align*}
\str{C_1} &= q_1\{q_2,q_3\}\{q_4,q_5\}, \\
\str{C_3} &= q_4\{q_3,q_5\}\{q_1,q_2\}
\end{align*}
don't satisfy~\eqref{case:nocommonapex}. It follows that the square $C_1+C_3$ must strongly link $H$: for example, it strongly links the square $q_1q_2q_4q_5$ in $H$. We conclude that $C_3$ must link $H$ in a star
$\str{C_3}=q_3O_3I_3$ with apex $q_3$, as claimed.

We now use the fact that both $C_1\cup C_3$ and $C_2\cup C_3$ are theta curves. The vertices $q_1$ and $q_2$ cannot both be common to every triangle in $H$ linking $C_3$,
because then the only triangle in $H$ linking $C_3$ would be $q_1q_2q_3$.  
So suppose without loss of generality that $q_1$ is not common to every triangle in $H$ linking $C_3$. Then $C_1\cup C_3$ must satisfy Case~\eqref{case:nocommonapex} of Theorem~\ref{thm:thetacurve}. However, the only star with apex $q_3$ that can satisfy~\eqref{case:nocommonapex} together with $\str{C_1}$ is $q_3\{q_1,q_2\}I=\str{-(C_1+C_2)}$, so it must be the case that $\str{C_3}=q_3\{q_1,q_2\}I$ too. But then
\[
\str{C_0+C_3} = \str{-(C_1+C_2)} = \str{C_3},
\]
and it follows that $C_0$ does not link $H$. This completes the proof.
\end{proof}

\begin{remark}
\label{rem:K4Kn-nocommonvertex-stars}
We may express the linking between $G$ and $H$ in
Case~\eqref{case:K4Kn-nocommonvertex} of Theorem~\ref{thm:K4Kn} in terms of stars in $G$ as follows. 

Notice that a triangle of $H$ links $G$ if and only if it has the form $\pm q_iq_jy$, for $i\neq j$ and $y\in I$. 
Consider the triangle $q_3q_2y$. This is linked positively by the triangles $C_3=p_1p_0p_2$ of $G$ and $-C_2=p_1p_0p_3$ of $G$, and hence links $G$ in the star $p_1\{p_0\}\{p_2,p_3\}$. Considering the triangles $yq_1q_3$ and 
$q_2q_1y$ in turn, we find that the triangles $\pm q_iq_jy$ link $G$ according to the following stars:
\begin{align*}
\str{q_3q_2y} &= p_1\{p_0\}\{p_2,p_3\}, \\
\str{yq_1q_3} &= p_2\{p_0\}\{p_1,p_3\}, \\
\str{q_2q_1y} &= p_3\{p_0\}\{p_1,p_2\}.
\end{align*}
\end{remark}

\section{Stars with no common apex}
\label{sec:nca}

\subsection{Introduction}

A key step in our characterisation of weakly linked embeddings of $G\cong K_m$ and $H\cong K_n$ is an analysis of the way in which two stars in a weakly linked embedding can meet. In Theorem~\ref{thm:nocommonapex}, we analyse a pair of stars that do not share a common apex. We begin with the following definitions.

\begin{defn}
\label{defn:commonapex}
Let $\Sigma_1=p_1O_1I_1,\Sigma_2=p_2O_2I_2$ be stars in $K_m$, with $m\geq 4$. 
\begin{enumerate}
\item
If $p_1\neq p_2$, then $\Sigma_1$ and $\Sigma_2$ are
\emph{mutually oriented} if $p_1\in O_2$ and $p_2\in O_1$.

\item
If there is a vertex $p$ of $K_m$ such that $\Sigma_1,\Sigma_2$ may be expressed in the form $\Sigma_i=pO_i'I_i'$ for each $i$, then $\Sigma_1$ and $\Sigma_2$ have \emph{a common apex}. Otherwise, we say that $\Sigma_1$ and $\Sigma_2$ have \emph{no common apex}.
\end{enumerate}
\end{defn}

\begin{remark} Suppose that $p_1\neq p_2$.  Then  $\eps_1\Sigma_1,\eps_2\Sigma_2$ are mutually oriented for a unique choice of signs $\eps_1,\eps_2\in\{\pm1\}$. Furthermore, if $\Sigma_1$ and $\Sigma_2$ are mutually oriented, then they have a common apex if and only if one of the following holds:
\begin{enumerate}
\item
$\Sigma_1$ is a fan with axis $p_1p_2$ (so that $\Sigma_1=p_1\{p_2\}I_1$ may be expressed in the form $\Sigma_1=p_2I_1\{p_1\}$).
\item
$\Sigma_2$ is a fan with axis $p_2p_1$ (so that $\Sigma_2=p_2\{p_1\}I_2$ may be expressed in the form $\Sigma_2=p_1I_2\{p_2\}$).
\item
There exists a vertex $p$ in $K_m$ such that $\Sigma_i$ is a fan with axis $pp_i$ for each $i$ (so that $\Sigma_i=p_iO_i\{p\}$ may be expressed in the form $\Sigma_i=p\{p_i\}O_i$ for each $i$).
\end{enumerate}
\end{remark}

\begin{thm}[Stars with no common apex]
\label{thm:nocommonapex}
Let $m$ and $n$ be positive integers with $m\geq 5$. Suppose that $G\cong K_m$ and $H\cong K_n$ are weakly linked graphs in $\real^3$, and $T_1$ and $T_2$ are triangles in $H$ linking $G$ in stars $\Sigma_1=p_1O_1I_1$ and $\Sigma_2=p_2O_2I_2$, respectively, which have no common apex.  Then (after possibly re-orienting $T_1$ and $T_2$ so that $\str{1}$ and $\str{2}$ are mutually oriented)
precisely one of the following holds:
\begin{enumerate}
\renewcommand{\labelenumi}{(C\arabic{enumi})}
\renewcommand{\theenumi}{C\arabic{enumi}}
\item\label{item:nca-Sigma1fan}
There is a vertex $p_3$ distinct from $p_1,p_2$ such that $I_1=\{p_3\}$ and $O_2=\{p_1,p_3\}$. 
\item\label{item:nca-Sigma2fan}
There is a vertex $p_3$ distinct from $p_1,p_2$ such that $I_2=\{p_3\}$ and $O_1=\{p_2,p_3\}$.
\item\label{item:nca-nofan}
There is a vertex $p_3$ distinct from $p_1,p_2$ such that $O_1=\{p_2,p_3\}$ and $O_2=\{p_1,p_3\}$.
\end{enumerate}
Thus, disregarding orientations, $\Sigma_1\cup\Sigma_2$ consists of all triangles sharing an edge with $T^*=p_1p_2p_3$, with the sole exception of $T^*$ itself in \eqref{item:nca-nofan}.
\end{thm}

\begin{remark}
\label{rem:nodisjointfans}
Note that if stars $\Sigma_1$ and $\Sigma_2$ satisfy one of~\eqref{item:nca-Sigma1fan}--\eqref{item:nca-nofan}, then at least one of them must be a proper star. Consequently, if triangles $T_1$ and $T_2$ in $H$ link fans in $G$ then the axes of the fans must intersect.
\end{remark}

Theorem~\ref{thm:nocommonapex} turns out to be easier to prove for $m\geq 6$ than for $m=5$.  However, for the sake of space, we present a single proof for $m\geq 5$. 
The proof breaks into two cases, according to the way in which $T_1$ and $T_2$ intersect. 
The case where they intersect in an edge was addressed by Theorem~\ref{thm:thetacurve}, and the case where they intersect in a single vertex is addressed by Proposition~\ref{prop:nocommonapex}.

\begin{prop}
\label{prop:nocommonapex}
Let $m$ and $n$ be positive integers with $m\geq 5$. Suppose that $G\cong K_m$ and $H\cong K_n$ are weakly linked graphs in $\real^3$. Let
$p_1,p_2$ be vertices of $G$ such that $p_1\neq p_2$, and let $C_1$ and $C_2$ be cycles in $H$ that intersect in a single vertex $q$, and link $G$ in mutually oriented stars $\Sigma_1=p_1O_1I_1$ and $\Sigma_2=p_2O_2I_2$, respectively. 
If $\Sigma_1$ and $\Sigma_2$ have no common apex, then 
they are described by one of conditions~\eqref{item:nca-Sigma1fan}--\eqref{item:nca-nofan}. 
\end{prop}

\begin{figure}
\begin{center}
\setlength{\radius}{1.5cm}
\begin{tikzpicture}[vertex/.style={circle,draw,minimum size = 2mm,inner sep=0pt},thick]
\foreach \i/\v in {0/q,1/{y_2},2/{x_2},3/{y_1},4/{x_1}}
   {\node (\i) at (-90+\i*72:\radius) [vertex] {};
    \node at (-90+\i*72:1.2*\radius) {$\v$};}
\foreach \i in {1,2,3,4}
       \draw (0) -- (\i);
\foreach \i/\j in {1/2,2/3,3/4}
     \draw (\i) -- (\j);
\foreach \t/\r/\i in {1/0.6/2,2.5/0.2/3,4/0.6/1,2.5/0.95/0}
       \node at (-90+\t*72:\r*\radius) {$C_\i$};
\end{tikzpicture}
\caption{The cycles $C_1$, $C_2$ and $C_3$ in the proof of Proposition~\ref{prop:nocommonapex}.}
\label{fig:nocommonapex}
\end{center}
\end{figure}
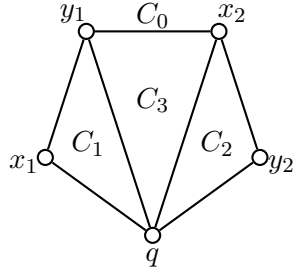

\begin{proof}
Without loss of generality we assume that $C_1$ and $C_2$ are triangles. Let $C_1=qx_1y_1$, $C_2=qx_2y_2$, and set $C_3=qy_1x_2$, $C_0=qy_2x_2y_1x_1$, so that 
\[
\sum_{i=0}^3 [C_i]=0
\]
in $H_1(H)$ (see Figure~\ref{fig:nocommonapex}). In addition, let $\Theta_1,\Theta_2$ be the theta curves $\Theta_1=C_1\cup C_3$, $\Theta_2=C_2\cup C_3$.

If $C_3$ does not link $G$, then $C_2'=C_2+C_3=qy_1x_2y_2$ links $G$ in $\Sigma_2$. The cycles $C_0$, $C_1$ and $C_2'$ together form a theta curve, and since $\Sigma_1$ and $\Sigma_2$ have no common apex we are in Case~\eqref{case:nocommonapex} of Theorem~\ref{thm:thetacurve}. Since also $\Sigma_1$ and $\Sigma_2$ are mutually oriented, it follows that Case~\eqref{item:nca-nofan} above holds.

Suppose then that $C_3$ links $G$ in a star $\Sigma_3$. We consider cases, according to which case of Theorem~\ref{thm:thetacurve} is satisfied by each of $\Theta_1$ and $\Theta_2$. 

\subsection*{Case 1: Both $\Theta_1$ and $\Theta_2$ satisfy  Case~\eqref{case:nocommonapex} of Theorem~\ref{thm:thetacurve}}

Then $\Sigma_3$ has no common apex with $\Sigma_1$ or $\Sigma_2$. Let $\Sigma_3=p_3O_3I_3$, where $p_3\neq p_1,p_2$.

By~\eqref{case:nocommonapex} applied to $\Theta_1$ there are two possible ways in which $\Sigma_1$ and $\Sigma_3$ can meet: either $O_1=O_3$, $p_1\in I_3$ and $p_3\in I_1$; or $I_1=I_3$, $p_1\in O_3$ and $p_3\in O_1$. Likewise, by~\eqref{case:nocommonapex} applied to $\Theta_2$ there are two possible ways in which $\Sigma_2$ and $\Sigma_3$ can meet: either $O_2=O_3$, $p_2\in I_3$ and $p_3\in I_2$; or $I_2=I_3$, $p_2\in O_3$ and $p_3\in O_2$. If $O_1=O_3$ and $p_1\in I_3$ then (since $p_1\in O_2$) both $O_2=O_3$ and $I_2=I_3$ are impossible, so it must be the case that $I_1=I_3$ and $p_1\in O_3$. By the same argument it must also be the case that $I_2=I_3$ and $p_2\in O_3$. Then by~\eqref{case:nocommonapex} it follows that 
\begin{align*}
O_1 &= \{p_2,p_3\}, &
O_2 &= \{p_1,p_3\}, &
O_3 &= \{p_1,p_2\},
\end{align*}
and $I_1=I_2=I_3=G-\{p_1,p_2,p_3\}$. This shows that~\eqref{item:nca-nofan} above holds.

\subsection*{Case 2: $\Theta_1$ satisfies Case~\eqref{case:commonapex} of Theorem~\ref{thm:thetacurve}, $\Theta_2$ satisfies Case~\eqref{case:nocommonapex}}

Then $\Sigma_1$ and $\Sigma_3$ have a common apex, and $\Sigma_2$ and $\Sigma_3$ are proper stars with no common apex.

Suppose first that $p_1$ cannot be chosen as the common apex of $\Sigma_1$ and $\Sigma_3$. Then we can write $\Sigma_1=p_3O_1'I_1'$, $\Sigma_3=p_3O_3I_3$, for some $p_3\neq p_1,p_2$. Since either $p_1$ or $p_3$ can be chosen as the apex of $\Sigma_1$, this star must be a fan with axis $\pm p_1p_3$ and therefore one of $O_1$ and $I_1$ must be equal to $\{p_3\}$. But $p_2\in O_1$ by hypothesis, so it must be $I_1$ that is equal to $\{p_3\}$; that is, $\Sigma_1=p_1O_1\{p_3\}=p_3\{p_1\}O_1$ and we have $O_1'=\{p_1\}$, $I_1'=O_1$. 

The star $\Sigma_3$ is proper, so $|I_3|\geq 2$ and hence $I_3$ contains some $r\neq p_1,p_3$. Then $r$ must belong to $I_1'=O_1=G-\{p_1,p_3\}$ also, which implies $I_1'\cap I_3\neq \emptyset$. It now follows by~\eqref{case:commonapex} applied to $\Theta_1$ that $O_1'\cap O_3$ must be empty, and therefore $p_1\in I_3$. Now recall that $p_1\in O_2$. This means that $O_2\neq O_3$ and $I_2\neq I_3$, in contradiction with Case~\eqref{case:nocommonapex}. We conclude that $p_1$ must be the common apex of $\Sigma_1$ and $\Sigma_3$.

Accordingly, let $\Sigma_3=p_1O_3I_3$ and consider $\Theta_2=C_2\cup C_3$. Since the apex $p_1$ of $\Sigma_3$ belongs to $O_2$, by~\eqref{case:nocommonapex} we must have $I_2=I_3$ and $O_2=\{p_1,p_3\}$, $O_3=\{p_2,p_3\}$ for some $p_3\in G$. Then since $O_1\cap O_3$ contains $p_2$ it is nonempty, so by~\eqref{case:commonapex} applied to $\Theta_1$ we must have $I_1\cap I_3=\emptyset$. But $I_3=G-\{p_1,p_2,p_3\}$ and $p_1,p_2\notin I_1$, so the only possibility is  $I_1=\{p_3\}$. We conclude that~\eqref{item:nca-Sigma1fan} holds.

\subsection*{Case 3: $\Theta_1$ satisfies Case~\eqref{case:nocommonapex} of Theorem~\ref{thm:thetacurve}, $\Theta_2$ satisfies Case~\eqref{case:commonapex}}

Reversing the roles of $\Theta_1$ and $\Theta_2$ in Case~2, we conclude that $\Sigma_1$ and $\Sigma_2$ satisfy~\eqref{item:nca-Sigma2fan}.

\subsection*{Case 4: Both $\Theta_1$ and $\Theta_2$ satisfy  Case~\eqref{case:commonapex} of Theorem~\ref{thm:thetacurve}}

Then $\Sigma_1$ and $\Sigma_2$ have a common apex with $\Sigma_3$. 
Since they do not share a common apex with each other, there must be vertices $r_1,r_2$ in $G$ such that $\Sigma_3$ is a fan with axis $\pm r_1r_2$, and $r_i$ can be chosen as the apex of $\Sigma_i$ for $i=1,2$.

Suppose first that $r_i=p_i$ for $i=1,2$. 
Without loss of generality we may assume that $\Sigma_3$ has axis $r_1r_2=p_1p_2$; that is, $\Sigma_3=p_1\{p_2\}I_3$, where $I_3=G-\{p_1,p_2\}$. Choose $p_3\in I_1$. Then $p_3\in I_3$ also, so triangle $T=p_1p_2p_3$ belongs to both $\Sigma_1$ and $\Sigma_3$. This means that $C_1'=C_1+C_3$ strongly links $T$, a contradiction.

It must therefore be the case that $r_i\neq p_i$ for some $i$.
Suppose without loss of generality that $r_1\neq p_1$. Let $p_3=r_1$, and let $\Sigma_1=p_3O_1'I_1'$, $\Sigma_3=p_3O_3I_3$. Then as in the second paragraph of Case~2 above we must have $\Sigma_1=p_3\{p_1\}O_1$; that is, $I_1=\{p_3\}$, $O_1'=\{p_1\}$, $I_1'=O_1=G-\{p_1,p_3\}$. 

By~\eqref{case:commonapex} applied to $\Theta_1$ at least one of $O_1'\cap O_3$ and $I_1'\cap I_3$ must be empty. The star $\Sigma_3$ is a fan with axis $\pm p_3r_2$, so one of $O_3$ and $I_3$  must equal $\{r_2\}$. If $I_3$ were equal to $\{r_2\}$ then 
(since $p_1$, $p_3$ and $r_2$ are distinct) we would have $r_2\in I_1'\cap I_3$ and $p_1\in O_1'\cap O_3$, contradicting the fact that at least one of
$O_1'\cap O_3$ and $I_1'\cap I_3$ must be empty. So it must instead be the case that $O_3=\{r_2\}$, and therefore $\Sigma_3=p_3\{r_2\}I_3$ for $I_3=G-\{r_2,p_3\}$. It now follows from~\eqref{case:commonapex} that $C_1'=C_1+C_3$ links the proper star $\Sigma_4=p_3\{p_1,r_2\}I_4$, for $I_4=G-\{p_1,p_3,r_2\}$. 

We now consider the theta curve $\Theta=C_1'\cup C_2$. If it were the case that $r_2\neq p_2$ then $\Sigma_2$ would be a fan with axis $\pm p_2r_2$ disjoint from the apex $p_3$ of $\Sigma_4$, which is impossible by Corollary~\ref{cor:disjointaxes} applied to $\Theta$. So we must instead have $r_2=p_2$, giving $\Sigma_3=p_3\{p_2\}I_3$, $\Sigma_4=p_3\{p_1,p_2\}I$ for $I=G-\{p_1,p_2,p_3\}$. 

Observe now that $C_3$ positively links the triangle $T=p_3p_2p_1$. Recall that $\Sigma_2=p_2O_2I_2$, with $p_1\in O_2$. If it were the case that $p_3\in I_2$ then $C_2$ would also positively link $T$, and then $C_2'=C_2+C_3$ would strongly link $T$. We must therefore have $p_3\in O_2$ instead, and hence $\{p_1,p_3\}\subseteq O_2$. It follows that $\Sigma_2$ and $\Sigma_4$ cannot have a common apex, so by Theorem~\ref{thm:thetacurve} applied to $\Theta$, $O_2$ must exactly equal $\{p_1,p_3\}$. We already have $I_1=\{p_3\}$, so this shows that~\eqref{item:nca-Sigma1fan} holds. This completes the proof.
\end{proof}

\begin{proof}[Proof of Theorem~\ref{thm:nocommonapex}]
Suppose that $\Sigma_1$ and $\Sigma_2$ have no common apex, and re-orient $T_1$ and $T_2$ (if necessary) so that $\Sigma_1$ and $\Sigma_2$ are mutually oriented. 
By Theorem~\ref{thm:notdisjoint} the triangles $T_1,T_2$ must intersect. 
If $T_1$ and $T_2$ meet in an edge, then $T_1\cup T_2$ forms a theta curve so $\Sigma_1$ and $\Sigma_2$ are described by Theorem~\ref{thm:thetacurve}. Since they have no common apex they must satisfy condition~\eqref{case:nocommonapex}, which co-incides with condition~\eqref{item:nca-nofan}.
Otherwise, $T_1$ and $T_2$ meet in a single vertex and the result follows from Proposition~\ref{prop:nocommonapex}.
\end{proof}

\section{Our main results}
\label{sec:KmKn}

We are now ready to complete our characterisation of weakly linked embeddings of $G\cong K_m$ and $H\cong K_n$. The first step is to prove the common vertex or common triangle dichotomy of Theorem~\ref{thm:dichotomy}, restated here as Theorem~\ref{thm:common-vertex-or-triangle}. This leads to two cases: one of $G$ and $H$ contains a common triangle (see Section~\ref{sec:commontriangle}), or both contain a common vertex (see Sections~\ref{sec:commonvertex} and~\ref{sec:commonvertex-realisation}).  In each case, we first determine the possible patterns of linking numbers; exhibit embeddings realising them; and then prove that our embeddings are weakly linked.

\begin{thm}
\label{thm:common-vertex-or-triangle}
Let $m\geq 5$ and $n\geq 4$. Suppose that $G\cong K_m$ and $H\cong K_n$ are weakly linked graphs in $\real^3$. If there is no vertex of $G$ common to all triangles of $G$ linking $H$, 
then there is a triangle $T^*$ in $G$ such that a triangle $T\neq T^*$ of $G$ links $H$ if and only if it shares an edge with $T^*$.
\end{thm}

\begin{proof}
Let $\linkedtriangles{G}$ be the set of oriented triangles in $G$ that link $H$, and let $\linkedtriangles{H}$ be the set of oriented triangles in $H$ that link $G$. By Theorem~\ref{thm:star} each triangle $T$ in $\linkedtriangles{H}$ links $G$ in a star $\str{T}$, and we let
\[
\mathcal{S}_G = \{\str{T}: T\in\linkedtriangles{H}\}.
\]
We claim that 
\[
\linkedtriangles{G} = \bigcup_{\Sigma\in\mathcal{S}_G} \Sigma. 
\]
Indeed, any triangle $S\in\str{T}\subseteq\mathcal{S}_G$ links the triangle $T$ of $H$, so belongs to $\linkedtriangles{G}$; and conversely, 
any triangle $S\in\linkedtriangles{G}$ must positively link some triangle $T\in\linkedtriangles{H}$ 
(for example, by subdividing a cycle $D$ in $H$ linking $S$, or because $S$ must link $H$ in a star $\str{S}$), and consequently belongs to $\str{T}$.

First suppose that there is no proper star in $\mathcal{S}_G$. Then every star in $\mathcal{S}_G$ is a fan, and by Remark~\ref{rem:nodisjointfans} the axes of any two such fans must intersect. Choose vertices $p_1$, $p_2$ in $G$ such that the fan with axis $p_1p_2$ belongs to $\mathcal{S}_G$, and note that the axis of any other fan in $\mathcal{S}_G$ must contain either $p_1$ or $p_2$. 
Since $p_1$ is not common to all triangles linking $H$, there must be 
a vertex $p_3$ of $G$ such that the fan with axis $p_2p_3$ belongs to $\mathcal{S}_G$; and since also $p_2$ is not common to all triangles linking $H$, there must be a vertex $p_4$ of $G$ such that the fan with axis $p_4p_1$ belongs to $\mathcal{S}_G$. But then 
the fan axes $p_2p_3$ and $p_4p_1$ are disjoint unless $p_3=p_4$. It follows that we must have $p_3=p_4$, and then $\mathcal{S}_G$ contains precisely the fans with axes $\pm p_1p_2$, $\pm p_2p_3$ and $\pm p_3p_1$. The triangle $T^*=p_1p_2p_3$ therefore satisfies the conclusion of the theorem.

Now suppose that there is $T_1\in\linkedtriangles{H}$ such that $\str{T_1}=p_1O_1I_1$ is a proper star. By assumption $p_1$ is not common to all triangles in $G$ linking $H$, so there is a triangle $T_2\in\linkedtriangles{H}$ such that $\str{T_2}=p_2O_2I_2$ has no common apex with $\str{T_1}$. Without loss of generality we may assume that $\str{T_1}$ and $\str{T_2}$ are mutually oriented, and then by Theorem~\ref{thm:nocommonapex} there is a vertex $p_3$ of $G$ such that $O_1=\{p_2,p_3\}$, and either $O_2=\{p_1,p_3\}$, or $\str{T_2}$ is a fan with axis $p_3p_2$. Let $T^*=p_1p_2p_3$, $I^*=G-\{p_1,p_2,p_3\}$, and for $i,j\in\{1,2,3\}$ define $I_{ij}=G-\{p_i,p_j\}$. 
We claim that, up to orientation, every star in $\mathcal{S}_G$ is equal to one of 
\begin{align*}
\Sigma_1 &= p_1\{p_2,p_3\}I^*, & 
\Sigma_2&=p_2\{p_3,p_1\}I^*, & 
\Sigma_3&=p_3\{p_1,p_2\}I^*, \\
\Sigma_{12} &= p_1\{p_2\}I_{12}, &
\Sigma_{23} &= p_2\{p_3\}I_{23}, &
\Sigma_{31} &=  p_3\{p_1\}I_{31}.
\end{align*}
It would then follow that $T^*$ satisfies the conclusion of the theorem. Note that under these conditions $T^*$ links $H$ if and only if one of the stars $\Sigma_{ij}$ belongs to $\mathcal{S}_G$.

So far we have $\str{T_1}$ equal to $\Sigma_1$, and $\str{T_2}$ equal to either $\Sigma_2$ or $-\Sigma_{23}$. 
The case $m\geq6$ is simpler than the case $m=5$, so we will assume for now that $m\geq6$ and address the case $m=5$ later.
Under the assumption $m\geq6$ we have $|I^*|\geq 3$.
If $\Sigma=pOI\in\mathcal{S}_G$ is a proper star with $p\neq p_1$, then $\Sigma$ must be one of $\pm\Sigma_2,\pm\Sigma_3$, by Theorem~\ref{thm:nocommonapex} applied to $\Sigma$ and $\str{T_1}=\Sigma_1$. In addition, if $\Sigma=pOI\in\mathcal{S}_G$ is a proper star with $p=p_1$, then $\Sigma$ must be $\pm\Sigma_1$, by Theorem~\ref{thm:nocommonapex} applied to $\Sigma$ and $\str{T_2}$, regardless of whether $\str{T_2}$ is equal to $\Sigma_2$ or $-\Sigma_{23}$. We conclude that, up to orientation, when $m\geq6$ the only proper stars that can belong to $\mathcal{S}_G$ are $\Sigma_1$, $\Sigma_2$ and $\Sigma_3$. 

Still assuming $m\geq6$, if $\Sigma\in\mathcal{S}_G$ is a fan with axis $pq$ disjoint from $p_1$ then by Theorem~\ref{thm:nocommonapex} we must have $\Sigma=\pm\Sigma_{23}$. The axis of any other fan in $\mathcal{S}_G$ must therefore meet $p_1$. Regardless of whether $\str{T_2}$ is equal to $\Sigma_2$ or $-\Sigma_{23}$, up to orientation the only other fan axes possible are $p_1p_2$ and $p_3p_1$, giving us $\Sigma_{12}$ or $\Sigma_{31}$: if $\Sigma_2\in\mathcal{S}_G$ then $p_3p_1$ is the only fan axis disjoint from $p_2$ satisfying Theorem~\ref{thm:nocommonapex} with respect to $\Sigma_2$, and $p_1p_2$ is the only axis meeting both $p_1$ and $p_2$; while if $-\Sigma_{23}\in\mathcal{S}_G$ then any axis must meet both $p_1$ and $p_2p_3$. Thus, the only possible stars in $\mathcal{S}_G$ are those listed above, and the theorem is proved for $m\geq6$.

We turn now to the case $m=5$. Then $|I^*|=2$, and we let $I^*=\{p_4,p_5\}$. The additional difficulty that arises in this case is that the fan with axis $p_4p_5$ and the proper stars $p_4\{p_2,p_3\}\{p_1,p_5\}$ and $p_5\{p_2,p_3\}\{p_1,p_4\}$ also satisfy Theorem~\ref{thm:nocommonapex} applied to $\Sigma_1$.

Suppose first that every proper star in $\mathcal{S}_G$ has apex $p_1$. Then $\str{T_2}$ is equal to $-\Sigma_{23}$. The only proper stars with apex $p_1$ that satisfy Theorem~\ref{thm:nocommonapex} with respect to $\Sigma_{23}$ are $\pm\Sigma_1$, so there can be no other proper star in $\mathcal{S}_G$. The axis of any other fan must meet $p_2p_3$; by Theorem~\ref{thm:nocommonapex} applied to $\Sigma_{1}$ the only possibilities are $\pm\Sigma_{12}$ and $\pm\Sigma_{31}$. Thus $T^*$ satisfies the required conditions.

Suppose finally then that there is a proper star in $\mathcal{S}_H$ with apex not equal to $p_1$. We could have chosen this star as $\str{T_2}$, so without loss of generality we may assume that $\str{T_2}=\Sigma_2$.
Any other proper star in $\mathcal{S}_G$ has no common apex with at least one of $\Sigma_1$ and $\Sigma_2$, and so must satisfy 
Theorem~\ref{thm:nocommonapex} with respect to one or both of $\Sigma_1$ and $\Sigma_2$. Up to orientation, the stars that are compatible with $\Sigma_1$ are
\begin{align*}
&& p_2|p_1p_3|p_4p_5, && p_4|p_2p_3|p_1p_5, \\
&& p_3|p_1p_2|p_4p_5, && p_5|p_2p_3|p_1p_4;
\end{align*}
while those that are compatible with $\Sigma_2$ are
\begin{align*}
&& p_1|p_2p_3|p_4p_5, && p_4|p_1p_3|p_2p_5, \\
&& p_3|p_1p_2|p_4p_5, && p_5|p_1p_3|p_2p_4.
\end{align*}
The only star that appears on both lists is $p_3\{p_1,p_2\}\{p_4,p_5\}=\Sigma_3$,
so we conclude that up to orientation the only proper stars that can belong to $\mathcal{S}_G$ are $\Sigma_1$, $\Sigma_2$ and $\Sigma_3$. 

By Theorem~\ref{thm:nocommonapex} applied to each of $\Sigma_1$ and $\Sigma_2$, if there is a fan other than $\pm\Sigma_{12}$, $\pm\Sigma_{23}$ and  $\pm\Sigma_{31}$ in $\mathcal{S}_G$ then it must have axis $\pm p_4p_5$. So suppose that there is a triangle in $H$ linking the fan $\Sigma_{45}=p_4|p_5|p_1p_2p_3$. Then this is in fact the only fan in $\mathcal{S}_G$, because $\Sigma_{12}$, $\Sigma_{23}$ and $\Sigma_{31}$ all have axes disjoint from $p_4p_5$. So up to orientation, the only stars that can belong to $\mathcal{S}_G$ are $\Sigma_i$ for $i=1,2,3$ and $\Sigma_{45}$. We show that this is impossible.

Label the triangles of the $K_4$ subgraph $K=\langle p_1,p_2,p_4, p_5\rangle$ such that
\begin{align*}
C_1 &= p_2p_4p_5, & C_2 &= p_5p_4p_1, \\
C_4 &= p_1p_2p_5, & C_5 &= p_4p_2p_1.
\end{align*}
Observe that each of these triangles belongs to at least one of $\pm\Sigma_1$, $\pm\Sigma_2$ and $\pm\Sigma_{45}$, and so links $H$. 
Since every triangle of $K$ links $H$ we must be in case~\eqref{case:K4Kn-commonvertex} of Theorem~\ref{thm:K4Kn}. It follows that there is a vertex $q$ of $H$, a sign $\eps\in\{\pm1\}$, and pairwise disjoint subsets $J_1,J_2,J_4,J_5$ of $H-\{q\}$ such that $C_i$ links $H$ in the star $\eps qP_iJ_i$ for $i=1,2,4,5$, where 
\[
P_i = H-\{q\}-J_i.
\]
Moreover, $J_i$ must be nonempty for each $i$, because otherwise $C_i$ does not link $H$. We may therefore choose $x_i\in J_i$ for $i=2,5$, to get a triangle $qx_2x_5$ in $H$ that links both $C_2$ and $C_5$ in $K$. But this is impossible, because none of the stars that can belong to $\mathcal{S}_G$ contains both of these triangles, so no triangle in $H$ can link both $C_2$ and $C_5$. It follows that the fan with axis $p_4p_5$ cannot belong to $\mathcal{S}_G$, and the theorem is proved.
\end{proof}

\subsection{Embeddings with a common triangle}
\label{sec:commontriangle}

In this section we analyse the case where at least one of $G$ and $H$ contains no vertex common to all triangles linking the other. Without loss of generality we may assume that this is $G$, and then by Theorem~\ref{thm:common-vertex-or-triangle}
there is a triangle $T^*$ in $G$ such that a triangle $T\neq T^*$ in $G$ links $H$ if and only if $T$ shares an edge with $T^*$. 

Theorem~\ref{thm:KmKn-triangle} shows that there are two possible patterns of linking numbers, according to whether or not $H$ contains a vertex common to all triangles linking $G$. We exhibit embeddings realising these in Figure~\ref{fig:KmKn}, and then prove that our embeddings are weakly linked in Theorem~\ref{thm:KmKn-triangle-weak}.

\begin{thm}
\label{thm:KmKn-triangle}
Let $m,n\geq 5$, and let $G\cong K_m$ and $H\cong K_n$ be weakly linked graphs in $\real^3$. Suppose that there is a triangle $T^*=p_1p_2p_3$ in $G$ such that a triangle $T\neq T^*$ in $G$ links $H$ if and only if $T$ shares an edge with $T^*$.
%, with the possible exception of $T^*$ itself. 
Let $X=G-\{p_1,p_2,p_3\}$, and for each $x\in X$ let
\begin{align*}
T_0(x) &= T^*=p_1p_2p_3, &  T_1(x) &= p_3p_2x, & T_2(x) &= xp_1p_3, & T_3(x) &= p_2p_1x.
\end{align*}
Then exactly one of the following holds:
\begin{enumerate}
\renewcommand{\labelenumi}{(D\arabic{enumi})}
\renewcommand{\theenumi}{D\arabic{enumi}}
\item
\label{case:KmKn-commonvertex}
There is a vertex $q$ of $H$ common to all triangles of $H$ linking $G$. Then there are pairwise disjoint sets $I_0$, $I_1$, $I_2$, $I_3$ such that $I_0\cup I_1\cup I_2\cup I_3=H-\{q\}$, and after reversing the orientation of $\real^3$ (if necessary), for each $x\in X$ and $0\leq i\leq 4$ the triangle $T_i(x)$ links $H$ in the star $qO_iI_i$, where
\[
O_i = H-\{q\}-I_i.
\]
Moreover, $I_i$ is nonempty for $1\leq i\leq 3$, and $I_0$ is nonempty if and only if $T^*$ links $H$. 
\item\label{case:KmKn-nocommonvertex}
There is no vertex of $H$ common to all triangles of $H$ linking $G$. Then $T^*$ does not link $H$, and there is a triangle $U^*=q_1q_2q_3$ of $H$ such that
a triangle $U$ of $H$ links $G$ if and only if $U$ shares exactly one edge with $U^*$. Let $Y=H-\{q_1,q_2,q_3\}$, and for each $y\in Y$ let
\begin{align*}
U_1(y) &= q_3q_2y, & U_2(y) &= yq_1q_3, & U_3(y) &= q_2q_1y.
\end{align*}
 Then after relabelling the $p_i$, $q_i$ and reversing orientation of $\real^3$ (if necessary), 
\begin{enumerate}
\item
\label{part:KmKn-nocommonvertex-G}
for each $x\in X$ the triangles $T_1(x),T_2(x),T_3(x)$ link $H$ in the stars
\[
q_1\{q_2,q_3\}Y, \qquad q_2\{q_1,q_3\}Y, \qquad q_3\{q_1,q_2\}Y;
\]
and
\item
\label{part:KmKn-nocommonvertex-H}
for each $y\in Y$ the triangles $U_1(y),U_2(y),U_3(y)$ link $G$ in the stars
\[
p_1X\{p_2,p_3\}, \qquad p_2X\{p_1,p_3\}, \qquad p_3X\{p_1,p_2\}.
\]
\end{enumerate}
\end{enumerate}
\end{thm}

\begin{proof}
Let $x_1,x_2\in X$. We begin by showing that 
\[
\link{T_i(x_1)}{D} = \link{T_i(x_2)}{D}
\]
for $1\leq i\leq 3$ and all cycles $D$ in $H$.

By symmetry, we may assume without loss of generality that $i=1$, so that $T_i(x_1)=T_1(x_1)=p_3p_2x_1$, $T_i(x_2)=T_1(x_2)=p_3p_2x_2$. Consider the $4$-cycle $C=p_2x_1p_3x_2$ in $G$. As a $1$-chain we have
\[
C = x_1p_3x_2+x_2p_2x_1.
\]
The triangles $x_1p_3x_2$, $x_2p_2x_1$ have no edge in common with $T^*$, so by hypothesis they do not link $H$. Hence in $H_1(\real^3-D)$ we have
\[
[C] = [x_1p_3x_2]+[x_2p_2x_1]=0+0=0.
\]
On the other hand, we may also write
\[
C=p_2x_1p_3+p_3x_2p_2=T_1(x_1)-T_2(x_2),
\]
and therefore 
\[
[T_1(x_1)]-[T_2(x_2)]=[C] =0.
\]
It follows that $\link{T_1(x_1)}{D} = \link{T_1(x_2)}{D}$ as claimed.

Fix $x\in X$. Since no triangle contained in $X$ links $H$ by hypothesis, it follows from the above that the linking between $G$ and $H$ is completely determined by the linking between $G'=\langle x,p_1,p_2,p_3\rangle$ and $H$. Since $G'\cong K_4$, this is given by Theorem~\ref{thm:K4Kn}, with $x$ in the role of $p_0$; that is, with $C_i=T_i(x)$ for $0\leq i\leq 3$. 

If case~\eqref{case:K4Kn-commonvertex} of Theorem~\ref{thm:K4Kn} holds, then (since $T_i(x')$ must link $H$ in the same star as $T_i(x)$ for $1\leq i\leq 3$ and all $x'\in X$) $G$ links $H$ according to Case~\eqref{case:KmKn-commonvertex} above. We note that $I_i$ is necessarily nonempty for $1\leq i\leq 3$, because $T_i(x)$ links $H$ for all $x\in X$ by hypothesis. Moreover, $T^*$ links $H$ if and only $I_0$ is nonempty, as given.

If Case~\eqref{case:K4Kn-nocommonvertex} of Theorem~\ref{thm:K4Kn} holds, then $G$ links $H$ according to~\eqref{part:KmKn-nocommonvertex-G}. To obtain part~\eqref{part:KmKn-nocommonvertex-H}, we replace $\{p_0\}$ with $X$ in the stars given in Remark~\ref{rem:K4Kn-nocommonvertex-stars}. 
\end{proof}

\begin{figure}
\setlength{\radius}{1.8cm}
\begin{center}
\begin{tabular}{cc}
% triangle in G, common vertex in H
\begin{tikzpicture}[vertex/.style={circle,draw,minimum size = 2mm,inner sep=0pt},thick,null/.style={minimum size = 0mm, inner sep = 0pt}]
% blue apices
\foreach \i in {1,2,3}
   {\node (p\i) at (60+\i*120:\radius) [vertex,fill=blue] {};
        \node at (60+\i*120:1.15*\radius) {$p_\i$};}
% lumped blue vertices
\begin{scope}[xshift=-0.3*\radius]
\foreach \j/\m/\t in {0/0.9/-5,1/1/{-10},2/1/0,3/1/-10}
    \node (x\j) at (-14+\t+\j*90:\m*0.15*\radius) [vertex] {};
\coordinate (bluecentre) at (0,0);
\end{scope}
% red apex
\node (q) at (0.3*\radius,0) [vertex,fill=red] {}; 
% red lumped vertices
% centres
\foreach \i/\t in {0/-30,1/30,2/120,3/240}    
   {\node (q\i) at (\t:1.35*\radius) {};}
% vertices at q0
\foreach \j/\m/\t in {0/1/15,1/1/{-10},2/1/-5,3/1/0}
    \path (q0) ++ (-14+\t+\j*90:\m*0.15*\radius) node (q0\j) [vertex] {};
% vertices at q1
\foreach \j/\m/\t in {0/1/40,1/1/{-10},2/1/0,3/1.1/10}
    \path (q1) ++ (-14+\t+\j*90:\m*0.15*\radius) node (q1\j) [vertex] {};
% vertices at q2
\foreach \j/\m/\t in {0/1/15,1/1/{-10},2/1/0,3/1/0}
    \path (q2) ++ (-34+\t+\j*90:\m*0.15*\radius) node (q2\j) [vertex] {};
% vertices at q3
\foreach \j/\m/\t in {0/1/15,1/1/{-10},2/1/0,3/1/0}
    \path (q3) ++ (-14+\t+\j*90:\m*0.15*\radius) node (q3\j) [vertex] {};
%%% crossing arcs
% q2 to q0 - anchors and arcs
\foreach \i in {0,1,2,3}
   {\coordinate (q2a\i) at (120:1.375*\radius+\i*0.05*\radius);
    \coordinate (q1t\i) at (30:1.9*\radius+\i*0.05*\radius);
    \coordinate (q0a\i) at (-30:1.375*\radius+\i*0.05*\radius);
    \draw [red] (q2a\i) to [out=45,in = 120] (q1t\i) to [out=-60,in=45] (q0a\i);}
% q3 to q1 - anchors
\foreach \i in {0,1,2,3}
   {\coordinate (q3a\i) at (-120:1.375*\radius+\i*0.05*\radius);
    \coordinate (q0t\i) at (-30:1.9*\radius+\i*0.05*\radius);
    \coordinate (q1a\i) at (30:1.375*\radius+\i*0.05*\radius);}
% white blocking arcs
    \draw [white,line width=16] (q3a1) to [out=-45,in = -120] (q0t1) to [out=60,in=-45] (q1a1);
    \draw [white,line width=16] (q3a2) to [out=-45,in = -120] (q0t2) to [out=60,in=-45] (q1a2);
% red arcs
\foreach \i in {0,1,2,3}
  \draw [red] (q3a\i) to [out=-45,in = -120] (q0t\i) to [out=60,in=-45] (q1a\i);
% red circles
\foreach \r in {1.36,1.41,1.31,1.27}
  \draw [color=red] (0,0) circle (\r*\radius);
% red under edges 
%\foreach \j in {0,1,2,3,}
%        \draw [color=white,line width=6] (q) -- (q0\j);
\foreach \j in {0,1,2,3}
        \draw [color=red] (q) -- (q0\j);
%% blue arcs
%draw overarc in white first
\draw [color=white,line width=8] (-50:\radius) arc (-50:0:\radius);
% then blue arcs
\foreach \i in {1,2,3}
    \draw [color=blue] (p\i) arc (60+\i*120:180+\i*120:\radius);
%% red edges
% draw each first in thick white
\foreach \i in {1,2,3}
    {\foreach \j in {0,1,2,3,}
        \draw [color=white,line width=8] (q) -- (q\i\j);}
% draw each in red
\foreach \i in {1,2,3}
    {\foreach \j in {0,1,2,3}
        \draw [color=red] (q) -- (q\i\j);}
% blue edges
% draw each first in thick white
\foreach \i in {1,2,3}
    {\foreach \j in {0,1,2,3}
        \draw [color=white,line width=8] (p\i) -- (x\j);}
% draw each in blue
\foreach \i in {1,2,3}
    {\foreach \j in {0,1,2,3}
        \draw [color=blue] (p\i) -- (x\j);}
% blue lumping circle
\filldraw [color=white,draw=blue] (bluecentre) circle (0.22*\radius);
% red lumping circles
\foreach \i in {0,1,2,3} 
{\filldraw [color=white,draw=red] (q\i) circle (0.25*\radius);
\node at (q\i) {$I_{\i}$};}
% label on red apex
\node [above] at (q) {$q$};
\node at (bluecentre) {$X$};
\end{tikzpicture} &
%%%%%%%%%%%%%%%%%%%%%%%%%%%%%%%%%%%%%%%%%%%%%%%%%%%%%%%%%%%%
%%%%%%%%%%%%%%%%%%%%%%%%%%%%%%%%%%%%%%%%%%%%%%%%%%%%%%%%%%%%
% no common vertex in either G or H
%\setlength{\radius}{1.9cm}
\begin{tikzpicture}[vertex/.style={circle,draw,minimum size = 2mm,inner sep=0pt},thick]
% red circle
\draw [color=red] (0,0) circle (1.3*\radius);
% apices
\foreach \i in {1,2,3}
   {\node (p\i) at (60+\i*120:\radius) [vertex,fill=blue] {};
    \node (q\i) at (-120+\i*120:1.3*\radius) [vertex,fill=red] {};
    \node at (60+\i*120:1.15*\radius) {$p_\i$};
    \node at (-120+\i*120:1.47*\radius) {$q_\i$};}
% lumped blue vertices
\begin{scope}[xshift=-0.3*\radius]
\foreach \j/\m/\t in {0/1/15,1/1/{-10},2/1/0,3/1/0}
    \node (x\j) at (-44+\t+\j*90:\m*0.15*\radius) [vertex] {};
\coordinate (bluecentre) at (0,0);
\end{scope}
% lumped red vertices
\begin{scope}[xshift=0.3*\radius]
\foreach \j/\m/\t in {0/1.2/{-10},1/1.1/{-5},2/1/3,3/1/0,4/0.5/0}
    \coordinate (y\j) at (-44+\t+\j*72:\m*0.15*\radius);
\coordinate (redcentre) at (0,0);
\end{scope}
% blue edges
\foreach \i in {1,2,3}
    {\foreach \j in {0,1,2,3}
        \draw [color=blue] (p\i) -- (x\j);}
%% red edges
% draw each first in thick white
\foreach \i in {1,2,3}
    {\foreach \j in {0,1,2,3,4}
        \draw [color=white,line width=4] (q\i) -- (y\j);}
% draw each in red
\foreach \i in {1,2,3}
    {\foreach \j in {0,1,2,3,4}
        \draw [color=red] (q\i) -- (y\j);}
%% blue arcs
% draw each first in thick white
\foreach \i in {1,2,3}
    \draw [color=white,line width=8] (70+\i*120:\radius) arc (70+\i*120:170+\i*120:\radius);
% draw each in blue
\foreach \i in {1,2,3}
    \draw [color=blue] (p\i) arc (60+\i*120:180+\i*120:\radius);
% blue lumping circle
\filldraw [color=white,draw=blue] (bluecentre) circle (0.22*\radius);
\node at (bluecentre) {$X$};
% red lumping circle
\filldraw [color=white,draw=red] (redcentre) circle (0.22*\radius);
\node at (redcentre) {$Y$};
% invisible point to align the figures
\node (xyz) at (-90:2*\radius) {};
\end{tikzpicture}
\end{tabular}
\caption{Embeddings of $G\cong K_m$ (blue) and $H\cong K_n$ (red) realising Cases~\eqref{case:KmKn-commonvertex} (left) and~\eqref{case:KmKn-nocommonvertex} (right) of Theorem~\ref{thm:KmKn-triangle}.}
\label{fig:KmKn}
\end{center}
\end{figure}
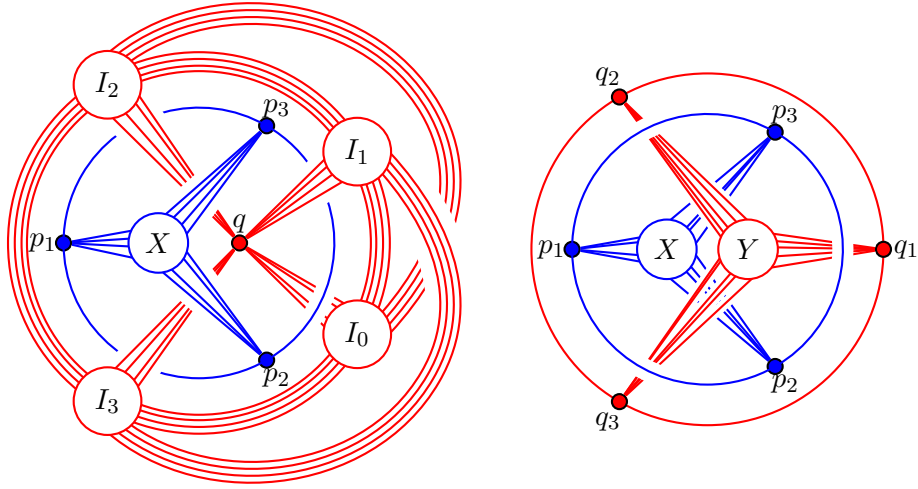

\begin{thm}
\label{thm:KmKn-triangle-weak}
The embeddings of Figure~\ref{fig:KmKn} realising Cases~\eqref{case:KmKn-commonvertex} and~\eqref{case:KmKn-nocommonvertex} of Theorem~\ref{thm:KmKn-triangle} are weakly linked.
\end{thm}

\begin{proof}
Let $T^*$ be the triangle $p_1p_2p_3$, and let $G'$ be $G$ minus the three edges $p_1p_2$, $p_2p_3$ and $p_3p_1$ of $T^*$. Then there is a $2$-sphere separating $G'$ from $H$, so $G'$ does not link $H$. Therefore any cycle $C$ in $G$ that links $H$ must use at least one edge belonging to $T^*$.

If $C$ uses all three edges of $T^*$, then we necessarily have $C=\pm T^*$. In the embedding of Figure~\ref{fig:KmKn} (left) $T^*$ does not link $H$, and in the embedding of Figure~\ref{fig:KmKn} (right) $T^*$ links $H$ in the star $qO_0I_0$, where $O_0=I_1\cup I_2\cup I_3$. In either case $C$ does not strongly link $H$, so we may assume in what follows that $C$ uses at most two edges of $T^*$. 

The edges of $T^*$ on $C$ must occur consecutively, so we may decompose $C$ as the concatenation $PQ$, where $P$ is a path in $T^*$ and $Q$ is a path in $G'$. By symmetry, we may assume without loss of generality that $P$ begins at $p_1$ and ends at $p_2$ (travelling anticlockwise if it has length $1$, and clockwise via $p_3$ if it has length $2$). Let $x_0$ be the first vertex of $X$ on $Q$, and let $R$, $\bar{R}$ be the paths $p_2x_0p_1$, $p_1x_0p_2$, respectively. Then we may decompose $C$ as $C=C_1+C_2$, where $C_1$ is the concatenation $PR$ and $C_2$ is the concatenation $\bar{R}Q$ (when $Q=R$ we have simply $C=C_1$). Then $C_2$ does not link $H$, because it is a cycle in $G'$, so for any cycle $D$ in $H$ we have
\[
\link{C}{D} = \link{C_1}{D}.
\]
To complete the proof we check that $C_1$ does not strongly link $H$, by verifying that it links a star in $H$.

If the path $P$ has length $1$ (that is, if it is simply the edge $p_1p_2$) then $C_1=p_1p_2x_0$ is the triangle $-T_3(x_0)$. In either embedding this links $H$ in a star: the star $qI_3O_3$ in Figure~\ref{fig:KmKn} (left),  and the star $q_3Y\{q_1,q_2\}$ in Figure~\ref{fig:KmKn} (right). On the other hand, if $P$ has length $2$ (that is, if $P$ is the path $p_1p_3p_2$), then $C_1$ links the star $q(I_1\cup I_2)(I_0\cup I_3)$ in Figure~\ref{fig:KmKn} (left),  and the star $q_3Y\{q_1,q_2\}$ in Figure~\ref{fig:KmKn} (right). In all cases $C_1$ links a star in $H$, so $G$ and $H$ are not strongly linked.
\end{proof}

\subsection{Embeddings with a common vertex in both $G$ and $H$}
\label{sec:commonvertex}

In this section we analyse the case where there is a vertex in each graph common to all triangles linking the other.
The linking between the two graphs is described by Theorem~\ref{thm:KmKn-pq}, and we exhibit a weakly linked embedding realising it in Figure~\ref{fig:pq-embedding}.

\begin{thm}
\label{thm:KmKn-pq}
Let $m,n\geq 5$, and let $G\cong K_m$ and $H\cong K_n$ be weakly linked graphs in $\real^3$. Suppose that there is a vertex $p$ of $G$ common to all triangles of $G$ linking $H$, and a vertex $q$ of $H$ common to all triangles of $H$ linking $G$. Then for some $2\leq\ell\leq\min\{m,n\}-1$, there exists 
\begin{itemize}
\item
a partition $\mathcal{X}=\{X_0,X_1,\ldots,X_{\ell-1}\}$ of $G'=G-\{p\}$, such that the triangle $pxy$ of $G$ links $H$ if and only if $x$ and $y$ belong to different parts of $\mathcal{X}$; and
\item
a partition $\mathcal{Y}=\{Y_0,Y_1,\ldots,Y_{\ell-1}\}$ of $H'=H-\{q\}$, such that the triangle $quv$ of $H$ links $G$ if and only if $u$ and $v$ belong to different parts of $\mathcal{Y}$.
\end{itemize}
Moreover:
\begin{enumerate}
\item
If $x_j\in X_j$, $x_k\in X_k$ for $j<k$, then $px_jx_k$ links $H$ in the star $qO_{jk}I_{jk}$, where
\begin{align*}
O_{jk} &= \bigcup_{i=j}^{k-1} Y_i, &
I_{jk} &= H'-O_{jk} = \left(\bigcup_{i=0}^{j-1}Y_i\right)\cup \left(\bigcup_{i=k}^{\ell-1}Y_i\right). 
\end{align*}
\item
If $y_j\in Y_j$, $y_k\in Y_k$ for $j<k$, then $qy_jy_k$ links $G$ in the star $pP_{jk}J_{jk}$, where
\begin{align*}
J_{jk} &= \bigcup_{i=j+1}^{k} X_i, &
P_{jk} &= G'-J_{jk} = \left(\bigcup_{i=0}^{j}X_i\right)\cup \left(\bigcup_{i=k+1}^{\ell-1}X_i\right).
\end{align*}
\end{enumerate}
\end{thm}
An embedding realising the linking of Theorem~\ref{thm:KmKn-pq} is described in Construction~\ref{cons:pq-embedding}, and the case $\ell=5$ is illustrated in Figure~\ref{fig:pq-embedding}. Note that the partitions $\mathcal{X}$ and $\mathcal{Y}$ are circularly rather than linearly ordered.

We will prove Theorem~\ref{thm:KmKn-pq} through a series of intermediate results. These will typically be proved under the hypotheses of Theorem~\ref{thm:KmKn-pq}.  To avoid repeating these, unless some other hypothesis is given,
\emph{we assume throughout this section that $G$ and $H$ are weakly linked.}
Our first step is to get our hands on the partition $\mathcal{X}$, which we will do by defining an equivalence relation $\simg$ on $G'$. The definition of $\simg$ depends only on the existence of the vertex $p\in G$ common to all triangles linking $H$, and not on the existence of the vertex $q\in H$ common to all triangles linking $G$. For full generality we therefore begin by assuming only the existence of $p$, and postpone introducing the hypothesis of the existence of $q$. Thus, unless some other hypothesis is given, \emph{we assume throughout this section that there is a vertex $p$ of $G$ common to all triangles of $G$ linking $H$}.

\begin{defn}
\label{defn:unlinkedrelation}
Let $G'=G-\{p\}$. We define a relation $\simg$ on the vertices of $G'$ by $x\simg y$ if and only if $x=y$, or $x\neq y$ and $pxy$ does not link $H$.
\end{defn}

We prove that $\simg$ is an equivalence relation on $G'$ in Lemma~\ref{lem:unlinkedequivalence} below. We 
will write $[x]$ for the equivalence class of $x\in G'$ with respect to $\simg$, and $\mathcal{X}$ for $\{[x]:x\in G'\}$, the set of equivalence classes of $\simg$. Note that $|\mathcal{X}|\geq 2$, because if $|\mathcal{X}|=1$ then $G$ does not link $H$. 

To prove Lemma~\ref{lem:unlinkedequivalence} and 
establish some other properties of $\simg$ we will repeatedly use the following lemma.

\begin{lem}
\label{lem:threepointsum}
Let $x,y,z$ be distinct vertices of $G'$, and let $D$ be a cycle of $H$. Then
\begin{equation}
\label{eq:threepointsum}
[pxy]+[pyz]+[pzx]=0
\end{equation}
holds in $H_1(\real^3-D)$. 
\end{lem}

\begin{proof}
The triangles $pxy$, $pyz$, $pzx$ and $zyx$ satisfy $pxy+pyz+pzx+zyx=0$ as $1$-chains in $G$, so in $H_1(\real^3-D)$ we have
\[
[pxy]+[pyz]+[pzx]+[zyx]=0.
\]
By assumption $p$ is common to all triangles of $G$ linking $H$, so $zyx$ does not link $H$ and therefore $[zyx]=0$. The lemma follows. 
\end{proof}

\begin{lem}
\label{lem:unlinkedequivalence}
The relation $\simg$ on $G'$ of Definition~\ref{defn:unlinkedrelation} is an equivalence relation.
\end{lem}

\begin{proof}
The relation $\simg$ is reflexive by definition, and it is symmetric because $\link{pyx}{D}=-\link{pxy}{D}$ for all $x\neq y$ in $G'$ and any cycle $D$ in $H$. To prove that $\simg$ is transitive, suppose that $x,y,z$ are distinct vertices of $G'$ such that $x\simg y$ and $y\simg z$. Let $D$ be a cycle of $H$.
Then $[pxy]=[pyz]=0$ in $H_1(\real^3-D)$, so by Lemma~\ref{lem:threepointsum} we have
\[
[pxz] = [pxy]+[pyz] =0+0=0
\]
also. Since this holds for any cycle $D$ in $H$ we conclude that $pxz$ does not link $H$, and therefore $x\simg z$.
\end{proof}

\begin{remark*}
In the embedding of Figure~\ref{fig:KmKn} (left), the vertex $q$ is common to all triangles of $H$ linking $G$. The equivalence classes of the corresponding relation $\simg$ defined on $H'=H-\{q\}$ are the sets $I_i$, for $0\leq i\leq 3$.
\end{remark*}

\begin{lem}
\label{lem:simimpliessamelinking}
Let $x$, $y\in G'$ with $x\simg y$, and let $D$ be a cycle in $H$. Then
\[
\link{pxz}{D} = \link{pyz}{D}
\]
for all $z\in G'$ with $z\neq x,y$.
\end{lem}

\begin{proof}
For any $z\in G'$ we have $[pxy]=0$ in equation~\eqref{eq:threepointsum}, and so $[pxz]=-[pzx]=[pyz]$. 
\end{proof}

\begin{lem}
Let $x$, $y\in G'$. Suppose that there is $z\in G'$ such that
\[
\link{pxz}{D} = \link{pyz}{D}
\]
for all cycles $D$ in $H$. Then $x\simg y$.  
\end{lem}

\begin{proof}
Let $D$ be a cycle of $H$. Applying Lemma~\ref{lem:threepointsum}, in $H_1(\real^3-D)$ we have
\[
[pxy] = [pxz]-[pyz]=0.
\]
Since this holds for any cycle $D$ in $H$ we conclude that $pxy$ does not link $H$, and therefore $x\simg y$. 
\end{proof}

We now introduce the hypothesis of the existence of $q$. Thus, unless some other hypothesis is given, \emph{we assume throughout the rest of this section that there is a vertex $q$ of $H$ common to all triangles of $H$ linking $G$}. 
Let $x,y\in G'$ be such that $x\not\simg y$. Then $pxy$ links $H$, so by Theorem~\ref{thm:star} it links $H$ in a star $qO_{xy}I_{xy}$ with apex $q$, because $q$ is common to all triangles of $H$ linking $G$. Note here that $\{O_{xy},I_{xy}\}$ is a partition of $H'=H-\{q\}$.

\begin{lem}
\label{lem:starequivalence}
Let $x,y\in G'$ such that $x\not\simg y$. Then the star $qO_{xy}I_{xy}$ depends only on the equivalence classes of $x$ and $y$. More precisely, if $x\simg z$ and $y\simg w$, then $O_{xy}=O_{zw}$ and $I_{xy}=I_{zw}$. 
\end{lem}

\begin{proof}
Since $x\simg z$, by Lemma~\ref{lem:simimpliessamelinking} we have $\link{pxy}{D}=\link{pzy}D$ for all cycles $D$ in $H$. It follows that $qO_{zy}I_{zy}=qO_{xy}I_{xy}$. Similarly, since $y\simg w$, we have $qO_{yz}I_{yz}=qO_{wz}I_{wz}$. The result now follows from the fact that if the triangle $pab$ links $H$ in the star $qOI$, then $pba=-pab$ links $H$ in the star $-qOI=qIO$; that is, $qO_{ba}I_{ba}=qI_{ab}O_{ab}$.
\end{proof}

Our next step is to establish the cyclic ordering of $\mathcal{X}$, the set of equivalence classes of $\simg$. We do this below by introducing a method of cyclically ordering triples of points in $G'$. This will be well defined on equivalence classes, and we will show that we can use it to cyclically order them.

Let $(x,y,z)$ be an ordered triple of points in $G'$ such that $x\not\simg y\not\simg z\not\simg x$. Consider $K_4=\langle p,x,y,z\rangle$, with the faces labelled and oriented such that
\begin{align*}
C_0 &= zyx, & C_1 &= pyz, & C_2 &= pxz, & C_3 &= pyz.
\end{align*}
Note that $\sum_i C_i=0$ as a $1$-chain in $G$. Since $q$ is common to all triangles of $H$ linking $K_4$ the linking between $K_4$ and $H$ is described by Case~\eqref{case:K4Kn-commonvertex} of Theorem~\ref{thm:K4Kn}. Furthermore $xyz$ does not link $H$, and the other three triangles all do because
$x\not\simg y\not\simg z\not\simg x$, 
so exactly one of the following holds:
\begin{enumerate}
\renewcommand{\labelenumi}{(\alph{enumi})}
\renewcommand{\theenumi}{(\alph{enumi})}
\item
\label{case:xyz+ve}
the sets $O_{xy}$, $O_{yz}$, $O_{zx}$ are a partition of $H'$, and
\begin{align*}
I_{xy} &= O_{zx}\cup O_{yz}, &
I_{yz} &= O_{xy}\cup O_{zx}, &
I_{zx} &= O_{yz}\cup O_{xy}; 
\end{align*} 
or
\item
\label{case:xyz-ve}
the sets $I_{xy}$, $I_{yz}$, $I_{zx}$ are a partition of $H'$, and
\begin{align*}
O_{xy} &= I_{zx}\cup I_{yz}, &
O_{yz} &= I_{xy}\cup I_{zx}, &
O_{zx} &= I_{yz}\cup I_{xy}.
\end{align*} 
\end{enumerate}
We define
\[
\eps(x,y,z) = \begin{cases}
              +1 & \text{if case~\ref{case:xyz+ve} holds}, \\
              -1 & \text{if case~\ref{case:xyz-ve} holds}.
              \end{cases}
\]
By Lemma~\ref{lem:starequivalence} the value of $\eps(x,y,z)$ depends only on the equivalence classes of $x$, $y$ and $z$ with respect to $\simg$, so we may define $\eps$ on triples of distinct equivalence classes by
\[
\eps([x],[y],[z]) = \eps(x,y,z). 
\]
Observe that 
\[
\eps(x,y,z) = \eps(y,z,x) = \eps(z,x,y)
\]
since cyclically permuting $x,y,z$ does not change the triangles involved;
and
\[
\eps(x,y,z) = -\eps(y,x,z),
\]
since swapping $x$ and $y$ reverses the orientations of all the triangles, and $qO_{ba}I_{ba}=qI_{ab}O_{ab}$ for all $a\not\simg b$.

\begin{remark}
\label{rem:edgesum}
Observe that if $\eps(x,y,z)=1$, then
\[
O_{xz} = I_{zx} = O_{xy}\cup O_{yz}.
\]
Note also that $O_{xy}\cap O_{yz}=\emptyset$, so $\{O_{xy},O_{yz}\}$ is a partition of $O_{xz}$. 
\end{remark}

\begin{lem}
\label{lem:cycliconquads}
Let $x,y,z,w$ be distinct vertices in $G'$ such that no two belong to the same equivalence class. Suppose that $\eps(x,y,z)=\eps(x,z,w)=1$. Then $\{O_{xy},O_{yz},O_{zw},O_{wx}\}$ is a partition of $H'$, and $\eps(y,z,w)=\eps(y,w,x)=1$. 
\end{lem}

\begin{proof}
Since $\eps(x,y,z)=1$, the sets $O_{xy}$, $O_{yz}$, $O_{zx}$ are a partition of $H'$, and $I_{zx} = O_{yz}\cup O_{xy}$. Likewise $O_{xz}$, $O_{zw}$, $O_{wx}$ are a partition of $H'$, and $I_{xz} = O_{wx}\cup O_{zw}$. Then  $O_{zx}=I_{xz}= O_{wx}\cup O_{zw}$, and since $\{O_{zx},I_{zx}\}$ is a partition of $G'$, it must be the case that $\{O_{xy},O_{yz},O_{zw},O_{wx}\}$ is a partition of $G'$ too. In particular, $O_{yz}\cap O_{zw}=O_{wx}\cap O_{xy}=\emptyset$, so the ordered triples $(y,z,w)$ and $(y,w,x)$ must both satisfy case~\ref{case:xyz+ve} above.
\end{proof}

\begin{prop}
\label{prop:cyclicordering}
Suppose that $|\mathcal{X}|=\ell$. Then there is a bijection $i\mapsto X_i$ from $\{i: 0\leq i\leq \ell-1\}$ to $\mathcal{X}$ such that $\eps(X_i,X_j,X_k)=1$ for $i\neq j\neq k\neq i$ if and only if the strictly increasing permutation of $i,j,k$ is a cyclic permutation of $(i,j,k)$.
\end{prop}

\begin{proof}
Fix $x_0\in G'$, and let $\mathcal{X}'=\mathcal{X}-\{[x_0]\}$. Define a relation $\preceq$ on $\mathcal{X}'$ by $[y]\preceq[z]$ if and only if $[y]=[z]$, or $[y]\neq[z]$ and $\eps(x_0,y,z)=1$. We claim that $\preceq$ is a total order on $\mathcal{X}'$. 

The relation $\preceq$ is reflexive by definition. To prove that it is antisymmetric, observe that if $[y]\neq[z]$, then exactly one of $\eps(x_0,y,z)=1$ and $\eps(x_0,z,y)=1$ holds, so exactly one of $y\preceq z$ and $z\preceq y$ holds. This also shows that the relation $\preceq$ is connex\footnote{A binary relation $\bowtie$ on a set $A$ is \emph{connex} if for all $x,y\in A$, the condition $x\bowtie y$ or $y\bowtie x$ holds.}, 
so it only remains to prove that $\preceq$ is transitive. This follows from Lemma~\ref{lem:cycliconquads}. Suppose that $y\preceq z$ and $z\preceq w$ for $y,z,w$ belonging to distinct classes. Then $\eps(x_0,y,z)=\eps(x_0,z,w)=1$, so by Lemma~\ref{lem:cycliconquads} $\eps(y,w,x_0)=1$. But $\eps(x_0,y,w)=\eps(y,w,x_0)$, so $y\preceq w$.

For $1\leq i\leq\ell-1$ choose $x_i\in G'$ such that $i\mapsto [x_i]$ is an order preserving bijection from $(\{i:1\leq i\leq\ell-1\},\leq)$ to $(\mathcal{X}',\preceq)$. Let $X_i=[x_i]$ for $0\leq i\leq\ell-1$. Then $i\mapsto X_i$ is a bijection from $\{i: 0\leq i\leq \ell-1\}$ to $\mathcal{X}$, and we claim it satisfies the required condition. 

To prove this, it suffices to show that $\eps(x_i,x_j,x_k)=1$ whenever $i<j<k$. 
For $[y],[z]\in \mathcal{X}'$ write $[y]\prec[z]$ if $[y]\neq[z]$ and $[y]\preceq[z]$. If $i=0$ then $\eps(x_0,x_j,x_k)=1$ by definition of $\preceq$, because $[x_j]\prec [x_k]$ if and only if $j<k$. Otherwise, since $0<i<j<k$ we have $x_i\prec x_j\prec x_k$, so $\eps(x_0,x_i,x_j)=\eps(x_0,x_j,x_k)=1$. Then $\eps(x_i,x_j,x_k)=1$ by Lemma~\ref{lem:cycliconquads}, and we are done.
\end{proof}

We now define the sets $Y_i$ of Theorem~\ref{thm:KmKn-pq}, and establish the structure of the stars $pO_{x_jx_{k}}I_{x_jx_{k}}$. 
As in the proof of Proposition~\ref{prop:cyclicordering}, for $0\leq i\leq \ell-1$ choose $x_i\in G'$ such that $X_i=[x_i]$. Let $Y_i=O_{x_ix_{i+1}}$ (subscripts on $x$ taken mod $\ell$), and set $\mathcal{Y}=\{Y_i:0\leq i\leq\ell-1\}$. Then:

\begin{prop}
\label{prop:G-outsets}
The set $\mathcal{Y}$ is a partition of $H'$, and if $j<k$ then
\begin{equation}
\label{eq:outsets}
O_{x_jx_k} = \bigcup_{i=j}^{k-1} Y_i.
\end{equation}
Consequently
\[
I_{x_jx_k} = H'-O_{x_jx_k} = \left(\bigcup_{i=0}^{j-1}Y_i\right)\cup
                             \left(\bigcup_{i=k}^{\ell-1}Y_i\right).
\]
\end{prop}

\begin{proof}
Each set $Y_i$ is nonempty, because $x_i\not\sim x_{i+1}$ and so $O_{x_ix_{i+1}}\neq\emptyset$. We show that $Y_i\cap Y_j=\emptyset$ if $i\neq j$. 

Note we consider subscripts mod $\ell$.
Without loss of generality, assume $i<j$. If $j=i+1$ then $Y_i\cap Y_{i+1}=\emptyset$ follows from $\eps(x_i,x_{i+1},x_{i+2})=1$, so suppose $j>i+1$. Consider the $4$-tuple $(x_i,x_{i+1},x_j,x_{j+1})$. Then $\eps(x_i,x_{i+1},x_j)=\eps(x_i,x_j,x_{j+1})=1$, so $\{O_{x_ix_{i+1}},O_{x_{i+1}x_j},O_{x_jx_{j+1}},O_{x_{j+1}x_i}\}$ is a 
partition of $H'$ by Lemma~\ref{lem:cycliconquads}. In particular, $Y_i\cap Y_j = O_{x_ix_{i+1}}\cap O_{x_jx_{j+1}}=\emptyset$, as required. 

The proof of equation~\eqref{eq:outsets} is by induction on $k$, using Remark~\ref{rem:edgesum} for the inductive step. 
The case $k=j+1$ holds by definition of $Y_j$. If the equation is true for some $k>j$, then since $\eps(x_j,x_k,x_{k+1})=1$, for $k+1$ we have 
\[
O_{x_jx_{k+1}} = O_{x_jx_k}\cup O_{x_kx_{k+1}} 
    = \left( \bigcup_{i=j}^{k-1} Y_i\right) \cup Y_k
    = \bigcup_{i=j}^{k} Y_i.
\]

To complete the proof we must show that $\bigcup_{i=0}^{\ell-1}Y_i=H'$. Given $u\in H'$, consider the triangle $px_{\ell-1}x_0$, which links $H$ in the star $qO_{x_{\ell-1}x_0}I_{x_{\ell-1}x_0}$. If $u\in O_{x_{\ell-1}x_0}=Y_{\ell-1}$ we are done; and otherwise we must have $u\in I_{x_{\ell-1}x_0}=O_{x_0x_{\ell-1}}=\bigcup_{i=0}^{\ell-2}Y_i$. 
\end{proof}

Since $q$ is common to all triangles of $H$ linking $G$, as in Definition~\ref{defn:unlinkedrelation} and Lemma~\ref{lem:unlinkedequivalence} we may define an equivalence relation $\simh$ on $H'$ by $u\simh v$ if and only if $quv$ does not link $G$. We show that $\mathcal{Y}$ is the set of equivalence classes of $\simh$ on $H'$:

\begin{cor}
The set $\mathcal{Y}$ is the set of equivalence classes of $\simh$ on $H'$ defined by $u\simh v$ if and only if $quv$ does not link $G$.
\end{cor}

\begin{proof}
Let $u,v\in H'$, and suppose that $u\in Y_i$, $v\in Y_j$. If $i\neq j$ then $u\in O_{x_ix_{i+1}}$ but $v\notin O_{x_ix_{i+1}}$, so $quv$ links $px_ix_{i+1}$. Therefore $u\not\simh v$. On the other hand, if $i=j$ then by Proposition~\ref{prop:G-outsets} $u$ and $v$ belong to the same part of $\{O_{xy},I_{xy}\}$ for all $x,y\in G'$ with $x\not\simg y$, so $quv$ does not link $pxy$ for any $x,y\in G'$ and therefore $u\simh v$.
\end{proof}

To complete the proof of Theorem~\ref{thm:KmKn-pq}, we establish the structure of the stars linked by triangles in $H$. This is done by re-expressing the linking described by the stars $qO_{x_ax_b}I_{x_ax_b}$ in terms of stars in $G$. 

\begin{prop}
\label{prop:H-outsets}
Suppose that $y_j\in Y_j$, $y_k\in Y_k$. If $j<k$ then $qy_jy_k$ links $G$ in the star $pP_{jk}J_{jk}$ in $G$, where
\begin{align*}
J_{jk} &= \bigcup_{i=j+1}^{k} X_i, &
P_{jk} &= G'-J_{jk} = \left(\bigcup_{i=0}^{j}X_i\right)\cup \left(\bigcup_{i=k+1}^{\ell-1}X_i\right).
\end{align*}
\end{prop}

\begin{proof}
Let $x_a,x_b\in G$ be such that $x_a\in X_a$, $x_b\in X_b$ and $\link{px_ax_b}{qy_jy_k}=1$; that is, so that $y_j\in O_{x_ax_b}$ and $y_k\in I_{x_ax_b}$. If $a<b$ then by Proposition~\ref{prop:G-outsets} we have $y_j\in O_{x_ax_b}$ and $y_k\in I_{x_ax_b}$ if and only if $a\leq j<b$ and $b\leq k$, so $a\leq j<b\leq k$. Otherwise, if $b<a$ then by Proposition~\ref{prop:G-outsets} we have $y_j\in O_{x_ax_b}=I_{x_bx_a}$ and $y_k\in I_{x_ax_b}=O_{x_bx_a}$ if and only if $b\leq k<a$ and $j<b$, so $j<b\leq k<a$. Thus $\link{px_ax_b}{qy_jy_k}=1$ if and only if $b$ belongs to the interval $(j,k]$ and $a$ does not, and the result follows.
\end{proof}

\subsection{Realising Theorem~\ref{thm:KmKn-pq}}
\label{sec:commonvertex-realisation}

We now describe an embedding of $G$ and $H$ in $\real^3$ realising the linking described by Theorem~\ref{thm:KmKn-pq}. We will use co-ordinates $(z,t)$ for $\real^3$, where $z\in\complex$ and $t\in\real$. 

\begin{construction}
\label{cons:pq-embedding}
Let $\mathcal{X}=\{X_0,X_1,\ldots,X_{\ell-1}\}$,
$\mathcal{Y}=\{Y_0,Y_1,\ldots,Y_{\ell-1}\}$ be partitions of $G'$ and $H'$, respectively, where $\ell\geq 2$. 
Let $\zeta$ be the $(2\ell)$th root of unity $\zeta=e^{\pi i/\ell}$, and choose $\rho\in\real$ such that $\rho<|1-\zeta|/2$, so that the circles centred on $1$ and $\zeta$ with radius $\rho$ do not intersect. This choice also ensures that the circles do not contain $0$. Place $p$ at $(0,1)$ and $q$ at $(0,-1)$, and for $0\leq j\leq \ell-1$ 
\begin{itemize}
\item
place the points belonging to $X_j$ on the circle in the plane $t=-1$ with centre $\zeta^{2j}$ and radius $\rho$, so that they are equally spaced on this circle; and
\item
place the points belonging to $Y_j$ on the circle in the plane $t=+1$ with centre $\zeta^{2j+1}$ and radius $\rho$, so that they are equally spaced on this circle.
\end{itemize}
Connect $p$ to each vertex $x\in G'$ by a straight line, and similarly connect $q$ to each vertex $y\in H'$ by a straight line. No edge $px$ meets any edge $qy$, because the projections of these line segments into the plane $t=0$ meet only at $z=0$. To complete the embedding, join each pair of vertices in $G'$ by an embedded arc in the half space $t\leq -1$, and similarly join each pair of vertices in $H'$ by an embedded arc in the half space $t\geq 1$. 
\end{construction}

Figure~\ref{fig:pq-embedding} illustrates the embedding in the case $\ell=5$. We show that it realises the linking pattern of Theorem~\ref{thm:KmKn-pq} in Proposition~\ref{prop:pq-realised}, and then use Proposition~\ref{prop:commonvertex-weak} to show that it is in indeed weakly linked in Corollary~\ref{cor:pq-realised}. 

\begin{figure}
\begin{center}
\setlength{\radius}{3cm}
\setlength{\hgap}{0.25cm} % use for rho
\newcommand{\Rm}{0.9} % multiplier for red radius
\newcommand{\Bm}{1.5} % multiplier for blue radius
\begin{tikzpicture}[vertex/.style={circle,draw,minimum size = 2mm,inner sep=0pt},thick,scale=0.64]
%central vertices
\node (p) at (0.2*\radius,0) [vertex,fill=blue] {};
\node (q) at (-0.2*\radius,0) [vertex,fill=red] {};
% X_i,Y_i centres
\foreach \i in {0,1,2,3,4}
    {\coordinate (x\i) at (\i*72:\Bm*\radius);
     \coordinate (y\i) at (36+\i*72:\Rm*\radius);}
% Y_i points and red rays
\foreach \i/\count/\angle/\twist in {0/4/90/5,1/3/120/-15,2/3/120/0,3/4/90/5,4/5/72/0}
    \foreach \j in {1,...,\count}
        {\path (y\i) ++ (\twist+\j*\angle:\hgap) coordinate (y\i\j);
         \draw [red] (q) -- (y\i\j);}

% X_i points
\foreach \i/\count/\angle/\twist in {0/3/120/180,1/2/180/0,2/4/90/-10,3/5/72/-5,4/4/90/0}
    \foreach \j in {1,...,\count}
        {\path (x\i) ++ (\twist+\j*\angle:\hgap) coordinate (x\i\j);}

%%%% blue rays - first thick in white, then in blue
\foreach \i/\count in {0/3,1/2,2/4,3/5,4/4} % outer loop for white
    \foreach \j in {1,...,\count}
         \draw [white,line width=4] (p) -- (x\i\j);
\foreach \i/\count in {0/3,1/2,2/4,3/5,4/4} % outer loop for blue
    \foreach \j in {1,...,\count}
         \draw [blue] (p) -- (x\i\j);
% red and blue circles
\draw [color=white,line width=16] (0,0) circle (\Rm*\radius); % white circle for red
\foreach \r in {-0.05,0,0.05}
    {\draw [color=red] (0,0) circle (\Rm*\radius+\r*\radius);
     \draw [color=blue] (0,0) circle (\Bm*\radius+\r*\radius);}
% red outer arcs
\foreach \i in {0,...,5}
   {\coordinate (r-b\i) at (36+\i*72:1.1*\Rm*\radius); %beginning
    \coordinate (r-m\i) at (108+\i*72:1.3*\Rm*\radius); %middle
    \coordinate (r-e\i) at (180+\i*72:1.1*\Rm*\radius); % end
     \draw [white,line width=10] (r-b\i) to [out=100+\i*72,in=18+\i*72] (r-m\i) 
                       to [out=198+\i*72,in=100+\i*72] (r-e\i);  %white arc
     \draw [red,double,double distance=5pt] (r-b\i) to [out=100+\i*72,in=18+\i*72] (r-m\i) 
                       to [out=198+\i*72,in=100+\i*72] (r-e\i);} % red arcs
% for red torus knot
\draw [white,line width=10] (r-b1) to [out=100+72,in=18+72] (r-m1);
\draw [red,double,double distance=5pt] (r-b1) to [out=100+72,in=18+72] (r-m1); 
% blue outer arcs
\foreach \i in {0,...,5}
   {\coordinate (b-b\i) at (\i*72:1.05*\Bm*\radius); %beginning
    \coordinate (b-m\i) at (72+\i*72:1.3*\Bm*\radius); %middle
    \coordinate (b-e\i) at (144+\i*72:1.05*\Bm*\radius); % end
     \draw [white,line width=10] (b-b\i) to [out=64+\i*72,in=-18+\i*72] (b-m\i) 
                       to [out=162+\i*72,in=64+\i*72] (b-e\i);  %white arc
     \draw [blue,double,double distance=5pt] (b-b\i) to [out=64+\i*72,in=-18+\i*72] (b-m\i) 
                       to [out=162+\i*72,in=64+\i*72] (b-e\i);} % blue arcs
% for blue torus knot
     \draw [white,line width=10] (b-b1) 
                       to [out=64+72,in=-18+72] (b-m1);  %white arc
     \draw [blue,double,double distance=5pt] (b-b1) 
                       to [out=64+72,in=-18+72] (b-m1); % blue arc
% lumping circles
\foreach \i in {0,...,4}
    {\filldraw [color=red,fill=white] (y\i) circle (1.85*\hgap);
     \node at (y\i) {$Y_\i$};
     \filldraw [color=blue,fill=white] (x\i) circle (1.85*\hgap);
     \node at (x\i) {$X_\i$};}

% central node labels
\node [above right] at (p) {$p$};
\node [above left] at (q) {$q$};
\end{tikzpicture}
\caption{Construction~\ref{cons:pq-embedding} when $\ell=5$.}
\label{fig:pq-embedding}
\end{center}
\end{figure}
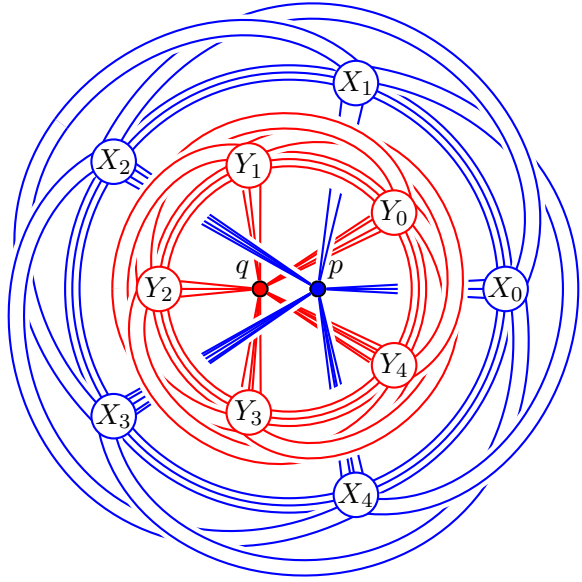

\begin{prop}
\label{prop:pq-realised}
The embedding of Construction~\ref{cons:pq-embedding} realises the linking pattern of Theorem~\ref{thm:KmKn-pq}.
\end{prop}

\begin{proof}
If $p$ is deleted from $G$ then there is a $2$-sphere separating $G'$ from $H$, so $G'$ does not link $H$. Similarly, if $q$ is deleted from $H$ then there is a $2$-sphere separating $H'$ from $G$, so $H'$ does not link $G$. Therefore $p$ is common to all triangles of $G$ linking $H$, and $q$ is common to all triangles of $H$ linking $G$. Let $x_j\in X_j$, $x_k\in X_k$, with $j<k$. We show that $px_jx_k$ links $H$ in the star $qO_{jk}I_{jk}$ of Theorem~\ref{thm:KmKn-pq}. This completely determines the linking between $G$ and $H$, because by Proposition~\ref{prop:H-outsets}
each triangle $quv$ in $H$ then links $G$ as described in Theorem~\ref{thm:KmKn-pq} also.

Let $C$ be the simple closed curve in $\real^3$ consisting of the line segment from $p=(0,1)$ to $(\zeta^{2j},-1)$, the arc of the unit circle in the plane $t=-1$ from $(\zeta^{2j},-1)$ to $(\zeta^{2k},-1)$ (taken in the positive direction, so that it contains the point $(\zeta^{k+j},-1)$), followed by the line segment from $(\zeta^{2k},-1)$ to $p$. There is an isotopy of $\real^3$ fixing $H$ and deforming $px_jx_k$ into $C$, so $\link{px_jx_k}{D}=\link{C}{D}$ for all cycles $D$ in $H$. We show that $C$ links $H$ in the star $qO_{jk}I_{jk}$.

The curve $C$ lies on the cone with apex $p$ that contains the unit circle in the plane $t=-1$. Let $F$ be the portion of this cone bounded by $C$. Then $F$ is a Seifert surface for $C$, so we may calculate $\link{C}{D}$ by counting signed intersections of $D$ with $F$. The only edges of $H$ which meet $F$ are edges of the form $qy_a$, with $y_a\in Y_a$ for $j\leq a<k$, and all such oriented edges meet $F$ with intersection number $+1$. It follows that a triangle $T$ of $H$ links $C$ if and only if it contains exactly one such edge, and the linking number is $+1$ if and only if $T$ orients the edge from $q$ to $y_a$. It follows that $C$, and hence $px_jx_k$, links $H$ in the star $qO_{jk}I_{jk}$, as required.
\end{proof}

To prove that the embedding of Construction~\ref{cons:pq-embedding} is weakly linked we will use the following proposition.

\begin{prop}
\label{prop:commonvertex-weak}
Let $m,n\geq 3$, and suppose that $G\cong K_m$ and $H\cong K_n$ are disjointly embedded in $\real^3$ such that
\begin{enumerate}
\item
there is a vertex $q$ of $H$ common to all triangles in $H$ that link $G$; and
\item
every triangle in $H$ that links $G$, links $G$ in a star. 
\end{enumerate}
Then $G$ and $H$ are not strongly linked.
\end{prop}

We note that the proposition may be used to give a second proof that the embedding of Figure~\ref{fig:KmKn} (left) is weakly linked.

\begin{proof}
Let $C$ be a cycle in $G$, and let $D=v_0v_1\ldots v_{k-1}$ be a $k$-cycle in $H$.  We will show that $C$ does not strongly link $D$. 
The argument is essentially identical to the proof of Lemma~\ref{lem:star-implies-weak}. 

If $q$ does not belong to $D$ then we decompose $D$ as the sum of the triangles $T_i=v_0v_iv_{i+1}$, for $1\leq i\leq k-2$. Since $q$ does not belong to $D$  but is common to all triangles in $H$ linking $G$ we have $\link{C}{T_i}=0$ for all $i$, and thus in the homology group $H_1(\mathbb{R}^3-C)$ we have
\[
[D] = \sum_{i=1}^{k-2} [T_i]=0.
\]
It follows that only cycles in $H$ that contain $q$ can link $G$.

Now suppose that $q$ belongs to $D$. 
By hypothesis and Lemma~\ref{lem:star-implies-weak} no triangle in $H$ strongly links $G$, so we may assume that $k\geq 4$.
Assume without loss of generality that $v_{0}=q$, and
let $T=v_0v_1v_{k-1}$, $D'=v_1v_2\ldots v_{k-1}$. Then $T$ is a triangle, $D'$ is a $(k-1)$-cycle, and $D=T+D'$ as $1$-chains in $H$. The cycle $D'$ does not contain $q$, so by the previous paragraph in $H_1(\mathbb{R}^3-C)$ we have
\[
[D]=[T]+[D']=[T]\in\{0,\pm1\}.
\]
Therefore $C$ does not strongly link $D$.
\end{proof}

\begin{cor}
\label{cor:pq-realised}
The embedding of Construction~\ref{cons:pq-embedding} is weakly linked.
\end{cor}

\begin{proof}
By Proposition~\ref{prop:pq-realised} the embedding of Construction~\ref{cons:pq-embedding} realises the linking pattern of Theorem~\ref{thm:KmKn-pq}, so $q$ is common to every triangle in $H$ linking $G$, and each triangle in $H$ that links $G$, links $G$ in a star. Therefore $G$ and $H$ are linked but not strongly linked, by Proposition~\ref{prop:commonvertex-weak}.
\end{proof}


\begin{thebibliography}{1}

\bibitem{conway-gordon1983}
J.~H. Conway and C.~McA. Gordon.
\newblock Knots and links in spatial graphs.
\newblock {\em J. Graph Theory}, 7(4):445--453, 1983.

\bibitem{drummondcole-odonnol2009}
Gabriel~C. Drummond-Cole and Danielle O'Donnol.
\newblock Intrinsically {$n$}-linked complete graphs.
\newblock {\em Tokyo J. Math.}, 32(1):113--125, 2009.

\bibitem{flapan2002}
Erica Flapan.
\newblock Intrinsic knotting and linking of complete graphs.
\newblock {\em Algebr. Geom. Topol.}, 2:371--380 (electronic), 2002.
\newblock E-print arXiv:math/0205231v1.

\bibitem{I3L}
Erica Flapan, Ramin Naimi, and James Pommersheim.
\newblock Intrinsically triple linked complete graphs.
\newblock {\em Topology Appl.}, 115(2):239--246, 2001.

\bibitem{RST}
Neil Robertson, Paul Seymour, and Robin Thomas.
\newblock Sachs' linkless embedding conjecture.
\newblock {\em J. Combin. Theory Ser. B}, 64(2):185--227, 1995.

\bibitem{sachs1983}
Horst Sachs.
\newblock On a spatial analogue of {K}uratowski's theorem on planar graphs ---
  an open problem.
\newblock In {\em Graph theory (\L ag\'ow, 1981)}, volume 1018 of {\em Lecture
  Notes in Math.}, pages 230--241. Springer, Berlin, 1983.

\end{thebibliography}
\end{document}